\documentclass[10pt]{article}
\usepackage{amsmath, amssymb, amsthm, amscd, amsrefs}
\usepackage{mathtools}
\usepackage{graphicx}
\usepackage{xcolor}
\graphicspath{{images/}}		
\usepackage{comment}
\usepackage{xfrac}
\usepackage{faktor}
\usepackage{anysize}
\marginsize{7em}{7em}{5em}{5em}

\usepackage{subfig}
\usepackage{graphicx}
\usepackage{float}
\usepackage{tikz}
\usepackage{tikz-cd}
\newcommand*{\jfrac}[2]{\genfrac{}{}{0pt}{}{#1}{#2}}

\usepackage{ytableau}
\usepackage{indentfirst}

  \newtheorem{manualtheoreminner}{Theorem}
\newenvironment{mnthm}[1]{%
  \renewcommand\themanualtheoreminner{#1}%
  \manualtheoreminner
}{\endmanualtheoreminner}

\newtheorem{manualcorollaryinner}{Corollary}
\newenvironment{mncor}[1]{%
  \renewcommand\themanualcorollaryinner{#1}%
  \manualcorollaryinner
}{\endmanualcorollaryinner}

\newcommand*{\bif}[1]{\textbf{\emph{#1}}}
\newtheorem{defn}{Definition}
\newtheorem{lem}{Lemma}

\newtheorem{thm}{Theorem}
\newtheorem*{thm*}{Theorem}
\newtheorem*{cor*}{Corollary}

\newtheorem*{pthm*}{Teorema}
\newtheorem{ex}{Example}
\newtheorem{prop}{Proposition}

\newtheorem{probl}{\textbf{{Problem}}}
\newtheorem*{prob*}{Problema}

\newtheorem{cor}{Corollary}

\newtheorem{quest}{Q}

\everymath{\displaystyle}
\definecolor{deeppink}{rgb}{1.0, 0.08, 0.58}
\definecolor{uibred}  {HTML}{db3f3d}
\definecolor{uibblue} {HTML}{4ea0b7}
\definecolor{uibgreen}{HTML}{789a5b}
\definecolor{uibgray} {HTML}{d0cac2}
\definecolor{uiblink} {HTML}{00769E}

\makeatletter
\let\old@rule\@rule
\def\@rule[#1]#2#3{\textcolor{rulecolor}{\old@rule[#1]{#2}{#3}}}
\makeatother
\definecolor{rulecolor}{named}{uibred}

\usepackage{titlesec}
\titleformat{\paragraph}
{\normalfont\normalsize\bfseries}{\theparagraph}{1em}{}
\titlespacing*{\paragraph}
{0pt}{3.25ex plus 1ex minus .2ex}{1.5ex plus .2ex}

\newtheorem{rem}{Remark}
\newcommand{\C}{\mathbb{C}}
\newcommand{\CC}{\widehat{\mathbb{C}}}
\newcommand{\R}{\mathbb{R}}
\newcommand{\N}{\mathbb{N}}
\newcommand{\Z}{\mathbb{Z}}

\newcommand{\s}{\mathbb{S}}
\newcommand{\rr}{\overline{\mathbb{R}}}

\newcommand{\cc}{\overline{\mathbb{C}}}

\newcommand{\CPu}{\mathbb{C}\mathbb{P}^1}
\newcommand{\RPu}{\mathbb{R}\mathbb{P}^1}

\newcommand*{\myov}[1]{\overbracket[1.1pt][0pt]{#1}}
\newcommand{\tred}[1]{\textcolor{red}{#1}}

\definecolor{ForestGreen}{RGB}{34,139,34}

\theoremstyle{remark}

\usepackage{ragged2e}

\usepackage[pdftex,plainpages=false,pdfpagelabels,colorlinks=true,citecolor=blue,linkcolor=uibred,urlcolor=uibred,filecolor=green,bookmarksopen=true]{hyperref} 

\makeatletter
\newcommand{\xdashrightarrow}[2][]{\ext@arrow 0359\rightarrowfill@@{#1}{#2}}
\newcommand{\xdashleftarrow}[2][]{\ext@arrow 3095\leftarrowfill@@{#1}{#2}}
\newcommand{\xdashleftrightarrow}[2][]{\ext@arrow 3359\leftrightarrowfill@@{#1}{#2}}
\def\rightarrowfill@@{\arrowfill@@\relax\relbar\rightarrow}
\def\leftarrowfill@@{\arrowfill@@\leftarrow\relbar\relax}
\def\leftrightarrowfill@@{\arrowfill@@\leftarrow\relbar\rightarrow}
\def\arrowfill@@#1#2#3#4{%
  $\m@th\thickmuskip0mu\medmuskip\thickmuskip\thinmuskip\thickmuskip
   \relax#4#1
   \xleaders\hbox{$#4#2$}\hfill
   #3$%
}
\makeatother

\newcommand*\colvec[3][]{
    \begin{pmatrix}\ifx\relax#1\relax\else#1\\\fi#2\\#3\end{pmatrix}
} 

\begin{document}

\title{\textbf{A Combinatorial Presentation for Branched Coverings of the 2-Sphere.}}

\author{Arcelino Lobato \thanks{no afilliation} }

\maketitle


\begin{abstract}{\hspace{-.55cm}\emph{William Thurston} (1946-2012) gave a combinatorial characterization for generic branched self-coverings 
of the 
two-sphere by associating a planar graph to them \cite{STL:15}. By generalizing the notion of local balancing, the author extends the Thurston result to encompass any branched covering of the two-sphere. As an application, we supply a lower bound for the number of equivalence classes of real rational functions for each given ramification profile. Furthermore, as a consequence, we obtain a new proof for a theorem \cite{MR1888795, MR2552110, MTV2} that corresponds to a special case of a reality problem in enumerative geometry which was known as the B. \& M. Shapiro Conjecture, now it is a theorem \cite{MR2552110}. The theorem version that we prove concerns generic rational functions, assuring that if all critical points of that function are real, then we can transform it into a rational map with real coefficients by post-composition with an automorphism of $\CPu$. The proof we present is constructive and founded on elementary arguments.}
\end{abstract}

\tableofcontents
\newpage

\section{Introduction} 



In the current paper, we give a combinatorial description for orientation-preserving branched coverings of the two-dimensional sphere via a cellular graph that captures their ramification profile. 
In order to get that, we address to the following realization problem:

\begin{quest}\label{realiz-q}
\textcolor{uibred}{What kinds of graphs drawn on a genus $g$ compact orientable surface would be a pullback by an appropriate ramified cover of a Jordan curve running through the ramification points in the two-sphere, and what combinatorial properties would they have? }
\end{quest}

This is an effort to realize 
the conceptual purpose of encoding 
sophisticated mathematical structures/ objects/ theories through simpler ones. The most impressive and successful achievements in this sense were the Grothendieck Theory of \emph{Dessins d'Enfants} (\emph{Children's Drawings}). That theory was the Grothendieck’s answer to the problem of determining which Riemann surfaces are defined, as an algebraic curve, over the field of algebraic numbers. More than that, it was the realization of the ambitious Grothendieck's program aiming to provide a complete description of the automorphism group of the field of algebraic numbers, the \emph{Absolute Galois Group}. He showed that those 
 Riemann surfaces that can be realized as algebraic curves defined over the algebraic numbers are precisely those which can be constructed from 
graphs drawn on the underlying topological surfaces. These graphs are under the constraint that they decompose the surface in a disjoint collection of topological disks (see \cite{MR1305390}). This type of graphs are called cell graphs (or, cellular maps (see Definition $\ref{cell-g}$)). 

Around 2010, Thurston gave a combinatorial description for the generic ramified self-covers of the Riemann sphere by using a $4$-regular oriented graph drawn in the sphere. 
That was his answer to the conundrum he brought up in \emph{determining the shape of a rational function}. More precisely, furnishing a simple model that could encode the branching information of a rational function and from which the function could be (re)constructed. 

As far as we know, the most distant ancestor to this idea of capturing the essence of a mapping, as especified above, by restricting it to a  graph (finite union of $0$ and $1$ cells) is the combination of the \emph{Alexander (trick) lemma} \cite{MR3203728} with the \emph{Schöenflies theorem} \cite{MR728227} that allows us to distinguishes homeomorphisms of a closed $2$-cell, up to isotopy, by its restriction to the boundary circle.

Thurston has proved the following theorem. 

\begin{mnthm}{T}[Thurston - \cite{STL:15}]\label{T}
A $4$-regular planar oriented graph $\Gamma$ with $2d-2$ vertices is equal to $f^{-1}(\Sigma)$ for some generic degree $d$ branched covering, $f:\s^2 \rightarrow\s^2$, and some \emph{Jordan curve}, $\Sigma\subset \s^2$, that it contains the critical values of $f$, if and only if :
\begin{itemize}
\item[$\boldsymbol{0. }$]{\textbf{cellular decomposition:} every face of $\Gamma$ is a Jordan domain (in particular $\Gamma$ is connected)}
\item[$\boldsymbol{1. }$]{\textbf{global balancing:} for any alternating white-blue coloring of the faces of $\Gamma$, there are $d$ white faces, and there are $d$ blue faces, and }
\item[$\boldsymbol{2. }$]{\textbf{local balancing:} for any oriented simple closed curve $\gamma$ in $\Gamma$ that is bordered by blue faces on the left and white on the right (except at the corners), there are strictly more blue faces than white faces on the left side of $\gamma$.}
\end{itemize}
\end{mnthm}

{A degree $d$ branched covering of $\s^2$, is said to be \emph{generic} when it has the maximum number of critical values (branch points). For a self-covering of $\s^2$ that number namely equals $2d-2$.  Alternatively, when each critical point has a ramification index $2$ and their images are pairwise distinct. From the \emph{Riemann-Hurwitz formula} \cite{lando:03} for a genus $g$ branched covering of degree $d$ the maximal number of branching points is equal to $2(d+g-1)$. } 


We may naturally consider the \emph{Thurston's balanced graphs}  graphs as generalized \emph{Grothendieck's Dessins d'Enfants}. From each dessin d'enfant, we readily built a balanced graph by just marking a new vertex to each children's drawing face and then connecting it to all vertices incident to the face containing it. Those balanced graphs are all the same up to isotopy relative to the vertices of the dessin d'enfant. Nevertheless, in this paper, we will not trail in the direction stroked by Grothendieck. Still, it is worth mentioning that the graphs we present seem to be the most promising candidates to support the generalization of Grothendieck's theory, as Shabbat has pointed out in \cite{SH19}, and are also a potential tool for the Hurwitz theory. In a series of forthcoming papers in preparation, we address the enumeration problem of balanced graphs and relate it to Hurwitz theory, moduli spaces of curves and complex analysis. A class of operations defined on balanced graphs is another topic we explore. 

The main result in this paper is a far-reaching generalization of that Thurston's Theorem ($\textrm{\ref{T}}$).
It encompasses every degree $d>0$ branched coverings of $\s^2$  by closed topological surfaces of any genus and with any \emph{admissible ramification profile} (see Section $\ref{cap-03}$).  
To achieve that, we established a new definition of \emph{local balance} capable of comprising all positive genus compact surfaces.

The full general version we prove reads as follows:
\begin{mnthm}{}[ $\textrm{\ref{teo-a}}$ - \textcolor{uibred}{General version of a theorem by Thurston}]
A genus $g$ cell graph $\Gamma\subset S_g$ 
is equal to $f^{-1}(\Sigma)$ for some degree $d$ branched covering, $f: S_g  \rightarrow\s^2$, and some \emph{Jordan curve}, $\Sigma\subset \s^2$, that it contains the critical values of $f$, if and only if it is a balanced graph (i.e., globally and locally balanced).   
\end{mnthm}

Two classes of \emph{cellular graphs} are defined. They are called \emph{pullback graphs} ($\ref{pullbackg}$) and the \emph{admissible graphs} ($\ref{adm-g}$). The  \emph{pullback graphs} are actually generalized \emph{Dessins d'Enfants}. They are cellular graphs obtained by lifting through a ramified covering a \emph{Jordan curve} that contains the branch points of that mapping. An \emph{admissible graph} encodes a recipe for constructing a branched covering of $\s^2$ (see $\ref{bcconstruction}$).  Thus, we reduce the problem of showing that \emph{balanced graphs} are preimage by branched coverings of special curves in $\s^2$ into the matter of ensuring that we can promote it to an \emph{admissible graph}. To resolve that problem, it suffices demonstrating that an enriched balanced graph, \emph{i.e.}, those ones with some vertices of valence equal to $2$ inserted (see Definition $\ref{def:enrich-g}$ ), admits a good \emph{vertex labelling}, turning it on an \emph{admissible graph}. In the generic planar case \emph{Thurston} achieves this by resorting to \emph{Cohomology}. Alternatively, we translate this issue into a \emph{graph theoretical} problem, leading to a more elementary approach. Thus, we give a solution to it (see $\ref{vertex-capacity-fn}$). 

A generalization for \emph{orientation preserving branched selfcovering} of $\s^2$ of the Theorem $\ref{T}$ was also obtained by \emph{J. Tomasini} \cite{Tomtese, Tomart} in his doctoral thesis. 

\emph{Tomasini} does not use those planar graphs introduced by Thurston. 
He had considered a star map 
made up of a collection of Jordan arcs connecting a chosen regular point of the branched covering to each of its critical values. He takes the preimage of that cellular graph to get a combinatorial object associated with the ramified covering as  \emph{Thurston} had proposed. The combinatorial model that he obtained, as introduced above, turns out to be a cellular bipartite graph. 
\emph{Tomasine} translated \emph{Thurston's balance conditions} to the class of cellular bipartite planar graphs he takes into account. Having set all of that, a general planar version of the Theorem $\textrm{\textrm{\ref{teo-a}}}$ 
was provided by him, together with some results concerning the decomposition of balanced graphs. His surgery operations follows that decompositions operations introduced by \emph{Thurston} in \cite{STL:15}.

A \emph{real globally balanced real graph} is a distinguished type of planar graph characterized by being isotopic to a globally balanced graph that is invariant by the complex conjugation, whose all of its vertices are in $\RPu$ and $\RPu$ is a cycle on it.

In this paper, we count the \emph{real globally balanced graph} with $2d$ faces for each possible valence profile of the vertices. By \emph{valence profile}, we mean the list of the valence of the vertices for a fixed enumeration of them. 

Thus, let $\Gamma$  be a {real globally balanced graph} with an enumeration $C=\{c_1, c_2, \cdots , c_n\}\subset\RPu$ of its corners, respecting the cyclic order 
in the counter-clockwise sense. 

The \bif{Kostka Number} of shape $2 \times (d - 1)$ and weight $\textbf{a}$, denoted $K(d,{\textbf{a}})$, is the number of all \emph{SemiStandard Young Tableau} of shape $2 \times (d - 1)$ and weight $\textbf{a}$. For the definition of \emph{SemiStandard Young Tableau} consult section $\ref{ssyt}$

The author shows that,
\begin{mnthm}{}[$\textrm{\ref{thm-b}}$]
For every $n\in \{2, 3, \cdots, 2d-2\}$ points in ${\RPu}$  and $\emph{\textbf{a}}= (a_1, a_2,\cdots , a_n)\in\N^{n}$, such that $\sum_{k=1}^{n} a_k = 2d-2$, there exist $K(d, \emph{\textbf{a}})$ standard real globally balanced graphs whose \emph{valence profile} is the 
integer vector $2\textbf{a} + 2$ \emph{(}see  $\ref{2.5-sect}$\emph{)}.
\end{mnthm}

When the weight vector is $\textbf{a}= (1, 1,\cdots , 1)\in\N^{2d-2}$ the number $K(d, \emph{\textbf{a}})$ is the $d$-th \emph{Catalan Number}, $C(d) := \frac{1}{d}{2d-2\choose d-1}$. A real globally balanced graph coming through Theorem $\textrm{\ref{thm-b}}$ from a weight vector $\textbf{a}= (1, 1,\cdots , 1)\in\N^{2d-2}$ has all its vertices with a valence $4$.

Then, as a corollary, we had achieved:

\begin{mncor}{}[$\textrm{\ref{cor-b1}}$]
For every $2d-2$ points in ${\RPu}$ there exist $C(d) := \frac{1}{d}{2d-2\choose d-1}$ real globally balanced graphs with these points as vertices with valence $4$ \emph{(}see  $\ref{2.5-sect}$\emph{)}. 
\end{mncor}
Another related result that we also show is the following:

\begin{mnthm}{}[$\textrm{\ref{teo-c}}$]
Every \emph{globally balanced real graphs} is \emph{locally balanced}.  
\end{mnthm}

Theorem $\textrm{\ref{teo-c}}$ says that every real globally balanced graph is actually a balanced graph. Thus, applying our theorem $\textrm{\textrm{\ref{teo-a}}}$, under a normalization, we construct $K(d, {\textbf{a}})$ ramified covers that are pairwise not equivalent, where the equivalence is the equality up to automorphisms of the codomain Riemann sphere. 

Therefore, 
\begin{mnthm}{}[$\textrm{\ref{teo-d}}$]
Given a integer $d\geq 2$, for every $n\in \{2, 3, \cdots, 2d-2\}$ points in $\RPu$ o be critical points and vector $\emph{\textbf{a}}= (a_1, a_2,\cdots , a_n)\in\N^{n}$, such that $\sum_{k=1}^{n} a_k = 2d-2$, there exists at least $K(d, \emph{\textbf{a}})$ \emph{real  rational functions}, up to automorphisms of the range Riemann sphere, with those prescribed $n$ critical points.
\end{mnthm}

In \cite{EGSV05}, Eremenko, Gabrielov, Shapiro and Vainshtein had proved that the number of equivalence classes counted in the previous theorem is exactly the lower bound we provide.

In particular, it follows that,
\begin{mncor}{}[$\textrm{\ref{cor-d1}}$]
For every $2d-2$ prescribed points in ${\RPu}$ there are at least $C(d) := \frac{1}{d}{2d-2\choose d-1}$ real rational functions whose those chosen points are its critical points \emph{(}see  $\ref{2.5-sect}$\emph{)}. 
\end{mncor}

{\justifying The problem of counting the equivalence classes of rational functions of $\CPu$ for a given set of critical points was considered previously by \emph{Eisenbud} \& \emph{Harris} in \cite{EiH:83} and by \emph{Lisa Goldberg} in \cite{Gold:91}. In both treatments, that problem is reduced to a problem in Schubert Calculus on Grassmanians. 

\emph{Goldberg} \cite[Theorem~1.3.]{Gold:91} established that for $2d - 2$ points 
 $\CPu$ the number of equivalent classes of degree $d$ rational functions with those prescribed points as its critical  points (therefore, all of them with ramification index $2$) is at most the $d$-th Catalan Number, $C(d)$. 
 
Based on Schubert's Calculus, Representation Theory, Fuchsian Differential Equations and KZ Equation Theory, I. Scherbak \cite{Scherbak:02}  derived a combinatorial formula for the number of equivalence classes of rational functions according to the multiplicities for the critical points. But that formula are not as simple as desired. Further, Osserman \cite{Oss:03}, studying the reducibility of moduli space of Linear Series with prescribed ramification on a smooth curve $C$ (over an algebraic closed field with characteristic $0$), settles a recursive formula for the number of equivalence classes of maps from a curve of an arbitrary genus to $\mathbb{P}$ with specified critical points and multiplicities. The recursion is based on the number of rational functions with only three branch points. 

{\rem For the sake of reference, we would like to point out that we may get recursive formulae for the Kostka numbers of shape 2 (d-1) by looking at the formulas derived in \cite{Oss:03} e \cite{Scherbak:02}  and the enumeration established in [6].}

A related problem is determining a constraint on the set of prescribed critical points in $\CPu$ such that each equivalence class of rational functions with that critical points may have a real representative, i.e., an element that is the quotient of two polynomials with real coefficients and no common zeros. This corresponds to the simplest case of the \emph{B.} \& \emph{M. Shapiro} conjecture, 
which declares \emph{\textbf{given $m\geq 2$ and $d\geq 2$, if for a list of $m$ degree $m+d-2$ polynomials with complex coefficients, $R = (f_1 (z), f_2 (z), \ldots , f_m (z))\in(\C[z])^{m}$ linearly independent, the \emph{Wronski determinant} of $R$ has only real zeros, then $X_{R} = \langle f_1(z), f_2(z), \cdots , f_m (z)\rangle_{\C}$ is a vector subspace of $\C_{d}[z]$, the linear space of complex polynomials of degree at most $d>0$, that has a basis in $\R_{d}[z]$. In other words, there is a matrix $A \in \rm{GL}(2, \C)$ that transforms the basis $(f_1 (z), f_2 (z), \cdots , f_m (z))$ into the basis $(A\cdot f_1 (z), A\cdot f_2 (z), \cdots , A\cdot f_m (z))\in (R_{d}[z])^{m}$ made up of polynomial with real coefficients.}} 

The \emph{Wronski determinant} is the polynomial  
\[\textrm{W}(f_1, f_2, \cdots , f_m) := \det\left( \left( \frac{d}{dt}\right)^{i-1} f_j(t)\right)_{i,j=1,\cdots ,m}\]

The degree of $\textrm{W}(f_1, f_2, \cdots , f_m)$ is at most $m(d-1)$. The simpler case above mentioned corresponds to the case $m=2.$ The relation with rational functions on the Riemann sphere is as follows: Each non-constant \emph{rational function} $R=\dfrac{f}{g}$ with $f$ and $g$ co-prime can be naturally associated with the vector subspace of $\C_{d}[z]$ spanned by $f$ and $g$, $\langle f, g\rangle_{\C}$. A change of base matrix of a two plane in $\C[z]$ is a element $A=\left(\begin{smallmatrix}a&b\\c&d\end{smallmatrix}\right)$ of $\textrm{GL}(2, \C)$ that determine the automorphism of $\CPu$ given by $\alpha(z) = \dfrac{az + b}{cz + d}$. The left action (by multiplication) of $\textrm{GL}(2, \C)$ on the space of two dimensional subspaces of $\C_{d}[z]$ (the Grassmaniann of $2$-planes in the linear space $\C_{d}[z]$, $G(2, \C_{d}[z])$) corresponds to the left action of the automorphism group of $\CPu$ by composition on the space of (irreducible) rational function of $\CPu$. 

Moreover, the zeros of the \emph{Wronskian} of $R$, $W_{R} = f'g-fg'$, are the critical points of $R$. Therefore, for $m = 2$, the conjecture is equivalent to asserting that\emph{ if a rational function $R$ is such that all of its critical points are real, i.e., it is contained in $\RPu$, then it can be transformed into a quotient of two co-prime polynomials with real coefficients by a post-composition with a linear fractional transformation.}

For a comprehensive exposition, further generalized forms of the Shapiro's conjecture, as well as other relevant enumerative geometric problems, consult \cite{Sot-rsc}.  

The combination of our Corollary $\textrm{\ref{cor-d1}}$ with \emph{Goldberg}'s result \cite[Theorem~1.3.]{Gold:91} culminates into a new proof (see Theorem $\ref{sha-conj}$) for the case $m=2$ of Shapiros's conjecture (\emph{Eremenko-Gabrielov-Mukhin-Tarasov-Varchenko Theorem} \cite{MR1888795},\cite{MR2552110}). 
That demonstration achieve here is a more constructive and elementary resolution that the previous proof provided.

\section{Preliminaries}

We're going to present some basic concepts and results in this section.

\subsection{Graph Theory}

\begin{defn}[Cell Graph]\label{cell-g}
A \emph{cell graph} $\Gamma$ into an oriented topological surface $S=|\Gamma|$ is the $1$-skeleton (i.e., the union of the $0$ and $1$ cells) of a cellular decomposition of the surface $S$. Respectively, we call a $0$, $1$ and $2$ cell of a cell decomposition of a surface by \emph{vertex}, \emph{edge} and \emph{face}.

A cell graph is oriented if the boundary of each face is an oriented cycle.
The set of vertices, edges and faces of a cell graph $\Gamma$ are denoted by $V(\Gamma)$, $E(\Gamma)$ and $F(\Gamma)$, respectively.
\end{defn}

Note that we are denoting the underlying surface of a cell graph by $|\Gamma|$.

One another fundamental notion is the following 
especial partition of the set of faces of a cell graph. 

\begin{defn}[\textcolor{deeppink}{A}-\textcolor{blue}{B} alternating face coloring]
Let \textcolor{deeppink}{A} and \textcolor{blue}{B} be two colors.
An \textcolor{deeppink}{A}-\textcolor{blue}{B} alternating face coloring for a cell graph $G$ is an assignement of the colours \textcolor{deeppink}{A} and \textcolor{blue}{B} to the faces of $G$ in a manner that:
\begin{itemize}
\item{adjacents faces possesses different color attached to them;}
\item{the faces kept on the left side of its boundary are all of the color \textcolor{deeppink}{A}.}
\end{itemize}
\end{defn}

\begin{defn}[positivity of a cycle]\label{set-upththm}
Let $\Gamma$ be a cell graph 
with an \textcolor{deeppink}{A}-\textcolor{blue}{B} alternating face coloring. 

Each cycle into $\Gamma$ \emph{(}i.e., a concatenation of edges of $\Gamma$ that forms a simple closed curve\emph{)} that keeps only \textcolor{deeppink}{A} faces on its left side is said to be a \emph{positive cycle} of $\Gamma$ regarding such a face coloring.

\end{defn} 

We now introduce a \emph{graph theoretical} technical apparatus that is crucial to the realizability of a \emph{balanced graph} (see $\ref{balance}$) as a \emph{pullback graph} (see $\ref{pullbackg}$).

\begin{defn}[multi-extremal chargeable graph]\label{ch-g}
A multi-extremal chargeable graph $C:=C[I,O]:=C[I,X$ $;O,Y]$ is a bipartite graph $G[X,Y]$ (the underlying graph of $C$) with two distinguished set of vertices, an input set $I\subset X$ and an output set $O\subset Y$, together with a nonnegative real-valued function $c:V(G)-(I\sqcup O)\rightarrow \R_{>0}$. $c$ is the vertex-capacity function of $C$ and its value on an vertex $v$ is the capacity of $v$. When is necessary to emphasize the capacity function we say that $C$ is a multi-extremal chargeable graph with capacity $c$.

The vertices in 
$V(G)-I\sqcup O$ are called interior vertices. 
We denote by $int.V(G):=V(G)-(I\sqcup O)$ the subset of interior vertices. The edges with endpoints in $int.V(G)$ 
are called interior edges.  
\end{defn}

\begin{defn}[edge-weighting on a graph]\label{e-w-g}
A edge-weighting on graph $G$ is a real function $w:E(G)\rightarrow \R$. A graph with a edge-weighting is a weighted graph.
\end{defn}

\begin{defn}[feasible weighting]\label{f-weight}
A edge-weighting $w$ on a multi-extremal chargeable graph $G$  with capacity $c$ is feasible if it satisfies the following additional constraints:
\begin{itemize}
\item[$\boldsymbol(1)$]{$w$ is a real estrictly positive function,i.e., $w(E(G))\subset \R_{>0}$;}
\item[$\boldsymbol(2)$]{$\sum_{x\in N_G (v)} w(\{x,v\})=c(v)$ for each interior vertex $v\in int.V(G)$.}
\end{itemize}
 
The sums \[|w|^{in}:=\sum_{x\in N_G (I)}\sum_{\jfrac{v\in I;}{\{v,x\}\in E(G)}} w(\{v,x\})\]
and
\[|w|^{out}:=\sum_{x\in N_G (O)}\sum_{\jfrac{v\in O;}{\{v,x\}\in E(G)}} w(\{v,x\})\]
are respectively the input value and output value of $w$.

A multi-extremal chargeable graph $G$ with a feasible weighting is called  multi-extremal weighted graph.
\end{defn}

\begin{prop}[charge conservation]\label{vertex-capacity-fn}
Let $N=N[I,X;O,Y]$ be a multi-extremal chargeable graph with constant capacity $M\in\R_{>0}$. Then for any 
feasible weighting $w$ on $N$ the input and output values are equal.
\end{prop}

\begin{proof}
The proof we are going to give will be by induction on the number of interior edges of the multi-extremal chargeable graph.

Let's start verifying the base case.

Let $N$ be a multi-extremal chargeable graph with constant capacity $M$ with only one interior edge $e\in E(N)$, $n$ initial edges and $m$ terminal edges. And let $w:E(N)\rightarrow \R_{>0}$ be a feasible weighting on $N$ assigning the weight $k>0$ to $e\in E(N)$.
Then,
\[|w|^{in} +k=M=k+ |w|^{out}\]
Thus, we have
\[|w|^{in}=|w|^{out}\]

\begin{figure}[H]
\begin{center}
\includegraphics[width=4cm]{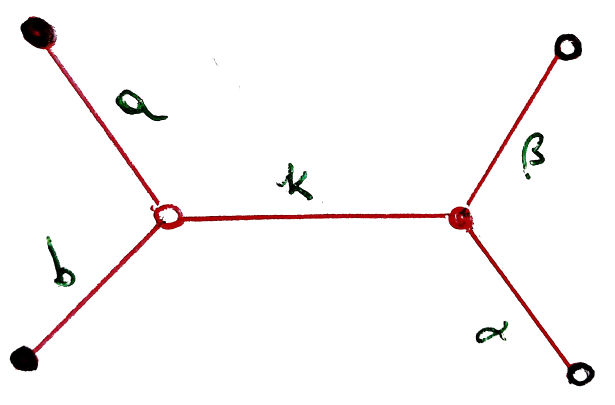}
\caption[chargeable graph]{$a+b+k=k+\alpha +\beta$}
\label{deg7pbg}
\end{center}
\end{figure}

Given $k>1$, we assume that for an arbitrary multi-extremal chargeable graph with constant capacity $M$ with $1<l\leq k$ interior edges, it is true that
\[|w|^{in}=|w|^{out}\]
for any feasible weighting $w:E(N)\rightarrow \R_{>0}$ on it.

Now, let $N'$ be a 
multi-extremal chargeable graph with constant capacity $M$ with $k+1$ interior edges and with a feasible weighting $w':E(N')\rightarrow \R_{>0}$ on $N'$. 

Let $E_1$ be a interior edge of $N'$ adjacent to at least one terminal edge of $N'$ and let $\epsilon_1=w'(E_1)>0$.

Let $a_1:=w'(A_1),a_2:=w'(A_2), \cdots ,a_p:=w'(A_p) $ be the list of the weights assigned by $w'$ to each terminal edge $A_h$ adjacent to $E_1$ with $h\in\{1,2,\cdots ,p\}$. 

There may exist more than one internal edge of $N'$ that is incident to the set of terminal edges $\{A_1, A_2,$ 
\\ $\cdots , A_p\}$. So, let $E_1, E_2, \cdots, E_u$ be those, possibly existing, edges with $\epsilon_l=w'(E_l)>0$ and let $\{b_{ij}\}_{j=1}^{u_i}$ the list of weights assigned by $w'$ to the interior edges of $N'$ that are incident to $E_{i}$.  And, let  $d_{j}>0$ for $j\in{1, 2, \cdots , r}$ be the list of weights assigned by $w'$ to all terminal edges of $N'$ different from those $E_i$ already considered.

Furthermore, let $F$ be the subgraph of $N'$ formed by the edges $\{A_1, A_2,  \cdots , A_p\}\cup\{E_1, E_2, \cdots , E_u\}$.

Then, the graph $N:=N'-F$ is a bipartite multi-extremal chargeable graph with constant capacity $M$ with $k+1-u\leq k$ interior edges and with a feasible weighting $w:w'|_{E(N)}:E(N)\rightarrow \R_{>0}$ on $N$. Thus, by the induction hypotesis,

 \begin{eqnarray}
 |w'|^{in}=|w|^{in}&=&|w|^{out}\\\nonumber
 &=& \left(\sum_{j=1}^{r}d_j\right) + \left(\sum_{i=1,j=1}^{u,u_i}b_{ij}\right)
\end{eqnarray}

But we also have
 \begin{eqnarray}
\left(\sum_{l=1}^{u}\epsilon_l\right) + \left(\sum_{i=1,j=1}^{u,u_i}b_{ij}\right) = M = \left(\sum_{l=1}^{u}\epsilon_l\right) +\left(\sum_{l=1}^{p}a_l\right)
\end{eqnarray}
Hence,
\[\sum_{i=1,j=1}^{u,u_i}b_{ij}=\sum_{l=1}^{p} a_l\]
Therefore,
\begin{eqnarray}
|w'|^{out} &=& \left(\sum_{j=1}^{r}d_j\right) + \left(\sum_{l=1}^{p}a_{l}\right)\\\nonumber
 &=&\left(\sum_{j=1}^{r}d_j\right) + \left(\sum_{i=1,j=1}^{u,u_i}b_{ij}\right)\\\nonumber
 &=& |w'|^{in}
\end{eqnarray}

\begin{figure}[H]
\begin{center}
\includegraphics[width=8.5cm, height=4cm]{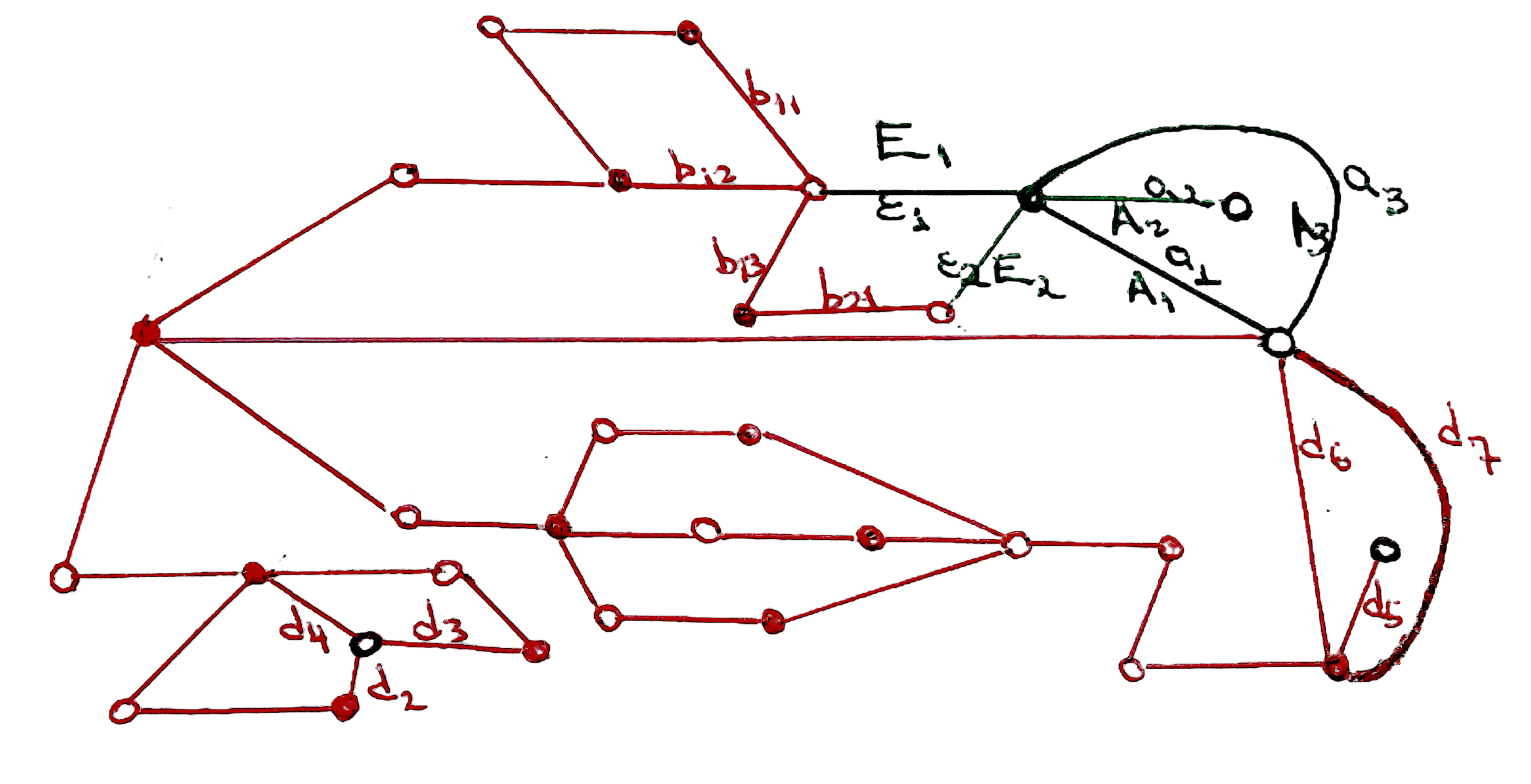}
\caption[weighted graph]{a weighted graph(with only the weights mentioned in the proof being visible)}
\label{deg7pbg}
\end{center}
\end{figure}
\end{proof}

\subsubsection{Machings on Graphs}

The Thurston key idea was to convert the realization issue $\ref{}$ into a graph theory matching problem. 

The following are the basic concepts and results we will need for from this subject.

\begin{defn}[matching]\label{match}
A \emph{matching} on a graph $G$ is a subset of edges $\mathcal{M}\subset E(G)$ that do not have vertices in common.
When a matching covers all vertices of $G$ it is called a \emph{perfect matching.}
\end{defn}

We refer to the problem of find out a matching on a bipartite graph as the \emph{Matching Problem}.

Note that if a bipartite graph $G := G[X, Y]$, with partition $\{X, Y\}$, admits a matching covering $X$, then the set of neighboring vertices ( from $Y$) for any subset of $X$, $S\subset X$, must contain at least as many vertices as S contains. The crux of \emph{Hall's theorem} is that such a condition suffices.
 
\begin{defn}[pontential mates]\label{pot-m}
Let $G=G[X, Y]$ be a bipartite graph. The set of potential mates for a collection of vertices in $X$, $S\subset V(G)\cap X$ is the set of neighbors of $S$ in $G$, i. e., 
\[N_{G}(S):=\{x\in V(G); \exists\; e\in E(G), \psi_G (e)=\{x,v\}\}.\]
\end{defn}

\begin{thm}[Hall's Merriage Theorem-\cite{BondyG}]\label{ksamento}
A bipartite graph $G := G[X, Y]$ has a matching which covers every vertex in $X$ if and only if
\[ |N_G(S)| \geq |S|\]
for all $S \subset X$.
\end{thm}
\begin{cor}[Perfect Matching Theorem]\label{cor-ksamento}
A bipartite graph $G := G[X, Y]$ has a perfect matching if and only if
$|X| = |Y|$ and $|N_G (S)| \geq |S|$ for all $S\subset X$.
\end{cor}

\subsection[Branched Coverings and Cell Graphs]{Branched Coverings and Cell Graphs}
\label{cap-03}
Let $S_g$ be a genus $g$ topological closed surface (i.e., compact and without boundary) and $f:S_g\rightarrow \s^2$ a \emph{degree $d\geq 2$ orientation-preserving branched covering} of $\s^2$ with $m$ critical values.


Let $\Sigma$ be an oriented Jordan curve passing through the critical values (branch points) of $f$, $R(f)$, and
 ${\Gamma}_{f} = \Gamma(f, \R) := f^{-1}(\Sigma)$ their inverse image by $f.$ $\Gamma$ is a cellularly embedded graph in $S_g$. $C(f)$ denotes the set of ramification (critical) points of $f$. We enumerate the $m=|R(f)|$ critical values of $f$ regarding the order that  $\Sigma$ runs through them in the positive sense according to the orientation of $\s^2$. The points in the fiber over a critical value are labelled by its label in that enumeration of $R(f)$. 
 
Henceforward, $f:S_g\rightarrow \s^2$ will be as above, unless otherwise indicated.

\begin{defn}[Post-critical curve]\label{def:pcc}
A \emph{post-critical curve} for $f$ is an isotopy class, relative to $R(f)$, of a \emph{Jordan curve} $\Sigma\subset\s^2$ passing through the critical values of $f$ into $\s^2.$ Such an isotopy class will be simply denoted by $\Sigma$, some representative of it.
\end{defn}



\begin{defn}[Pullback graph ($\mathfrak{G}$Dessin d'Enfant)]\label{pullbackg}
Given a post-critical curve $\Sigma$ for $f:S_g\rightarrow \s^2$. 

The \emph{pullback graph} of $f$ with respect to $\Sigma$, or simply, $\Sigma$-\emph{pullback graph} of $f$ is the isotopy class of ${\Gamma}:=f^{-1}(\Sigma)$ relative to $C_f$. 

The \emph{vertex set} of ${\Gamma}$ is the set 
$V({{\Gamma}})=f^{-1}({R_f})$. 

The \emph{edge set} of ${\Gamma}$, $E({{\Gamma}})$, is the set of the arcs into ${\Gamma}$ connecting pairs of points in $V({{\Gamma}})$. 

Given a Jordan curve $\Sigma\supset R(f)$ the set $f_{-1}(\Sigma)$ fills $S_g$, it separates $S_g$ into a finite disjoint union of 2-cells (topological disks).

And each connected component of $X-{\Gamma}$ is a \emph{face of} ${\Gamma}$ (or, boundary component). They are topological disks. 
\end{defn}

\begin{defn}[passport of a branched covering of $\s^2$]\label{passp}
Let $f:{S}_g\rightarrow \s^2$ be a degree $d$ branched covering of $\s^2$ with 
branch locus $\{w_1,w_2 , ..., w_m\}$.
The passport $\pi=\pi(f)$ of $f$ is the list of $m$ integer partition of $d$, $\pi(f)=[\pi_{1}, \pi_{2},..., \pi_{m}]$ with $\pi_{j}=[d_{(j, 1)}, d_{(j, 2)}\cdots, d_{(j, l_j)}]\in\Z/\mathcal{S}_{l_j}$, encoding the ramification profile of the fibers over the branch points. That is, the numbers $d_{(j, k)}$ are the multiplicities of the 
points in the fibre of $f$ over the critical value $w_j \in R(f)$.

Thus, for each $j\in\{1, 2, \cdots, m\}$, the data $\pi_{j} =[ d_{(j,1)}, d_{(j,2)}\cdots, d_{(j,l_j)}]$ is such that
\begin{itemize}
\item{$d_{(j, k)}\in\{1,2 ..., d\}$ for all $k\in\{1,2, \cdots, l_j\}$;}
 \item{for at least one $k\in\{1, 2, ..., l_j\}$, $d_{(j,k)}\neq 1$;}
 \item{$d=\sum_{k=1}^{l_j}d_{(j,k)}$}
\end{itemize} 
More conveniently, we will often write the partition $\pi_j$ as follows:
 \begin{eqnarray}\pi_j^{\ast} &=&(a_{1}^{p^{j}(a_1)}, a_2^{p^{j}(a_2)}, 3^{p^{j}(a_3)}, ..., n_j^{p^{j}(a_{n_j})})\end{eqnarray} with $1\leq a_1<  a_2< \cdots< a_{n_j}\leq d$ and $p^{j}(a)=|\{k\in\N; d_{(j,k)} = a\}|$.\hfill\\
\emph{(}Hence, $d=\sum_{n=1}^{n_j}a_n p^{j}(a_n)$ and $l_j =\sum_{n=1}^{n_j} p^{j}(a_n)$\emph{)}
\end{defn}

Given a Jordan curve $\Sigma\supset R(f)$, let  $\pi(f)=[\pi_1^{\ast} , \pi_{2}^{\ast}, ..., \pi_{m}^{\ast}]$ be the passport of $f$. 

For each 
point $w_j\in R(f)$ there are $l_j = \sum_{n=1}^{n_j} p^{j}(a_n)$ vertices in ${\Gamma}$ that are projected over it by $f$. We label each of those vertices with the number $j\in \{1, 2, \cdots, m\}.$ 
Among them there are $p^{j}(a_n)$ vertices of valence $2\cdot a_n$ for each $n\in\{1, 2, \cdots, n_j\}$. 
And $|V({{\Gamma}})|=\sum_{j=1}^{m}l_j.$

The vertices of valence strictly greater than $2$ are the critical points of $f$ and the other vertices are regular preimages of the critical values of $f$. 

Each connected component of $X-f^{-1}(\Sigma)$ is mapped injectively by $f$ over the $2$-cell in the left or right side of $\Sigma\subset\s^2$. Hence they are also topological disks. Each face has as its boundary a finite union of \emph{Jordan archs} connecting points of $f^{-1}(R_f)$. 
Therefore, \emph{Pullback graphs} are cell graphs.

Coloring the left side of $\Sigma\subset\s^2$ by the color \textcolor{deeppink}{A} and the right side of it by the color \textcolor{blue}{B}, an \textcolor{deeppink}{A}-\textcolor{blue}{B} alternating coloring for the faces of $\Gamma$ is determined.

\begin{figure}[H]
\centering
\subfloat{\tikz[remember
picture]{\node(1AL){\includegraphics[width=4cm]{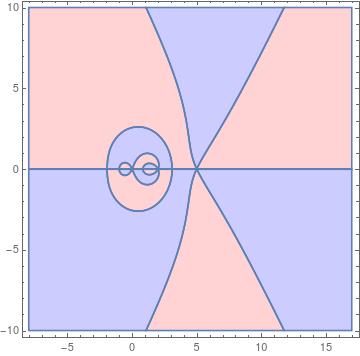}};}}%
\hspace*{1.5cm}%
\subfloat{\tikz[remember picture]{\node(1AR){\includegraphics[width=4cm]{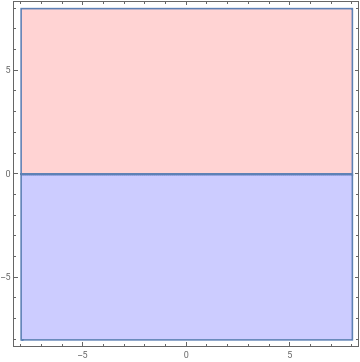}};}}
\caption[degree $7$ pullback graph]{$f(z)=$\small{$\dfrac{500- 50z- 1215z^2 +1388.5z^3 -674.5z^4  +166.5z^5 -20.5z^6 +z^7 }{-z^3 +z^4 }$}}
\end{figure}
\tikz[overlay,remember picture]{\draw[-latex,thick] (1AL) -- (1AL-|1AR.west)
node[midway,below,text width=1.5cm]{\hspace{.65cm}$f$};}


\begin{prop}[struture of a pullback graph]\label{prop-shape-graph} 
Let $f:X\rightarrow \s^2$ with passport $\pi_{f} = (\pi_{1}^{\ast}, \pi_{2}^{\ast}, ..., \pi_{m}^{\ast})$. Then, ${\Gamma}={\Gamma}_f(\Sigma):=f^{-1}(\Sigma)$ is a connected embedded graph on $X$ with $2d$ faces and $\sum_{j=1}^{m}l_j$ vertices such that each of its faces is a \emph{Jordan domain} containing on its boundary only one vertex corresponding to each critical value of $f$ with the labelings appearing cyclically ordered around it.

Furthermore, we have $p^{j}(a_n)$ vertices in $V({\Gamma})$ of valence $2a_n$ that are sent to  the critical value $w_j$ for each $j\in\{1, 2, ..., m\}$ and $n\in\{1, 2, ..., n_j\}.$
\end{prop}
\begin{proof}
Let $\Sigma$ be a $\emph{Jordan}$ curve passing through the critical values of $f$. $f|_{X-\{f^{-1}(R(f))\}}:X-\{f^{-1}(R(f))\}\longrightarrow \s^{2}-R(f)$ is a covering map. Let $\Sigma[w_{j}, w_{j+1}]$ be the Jordan arch of $\Sigma$ connecting $w_j$ to $w_{j+1}$, $\Sigma(w_{j}, w_{j+1}) =  \Sigma[w_{j}, w_{j+1}] - \{w_{j}, w_{j+1}\}$ 
and $a\in\Sigma(w_{j}, w_{j+1})$. 

The fiber of $f$ over $a$ contains $d$ distinct points. For each point $x_i \in f^{-1}(a)$ the inclusion map $Id_\Sigma(w_{j}, w_{j+1}):\Sigma(w_{j}, w_{j+1}) \hookrightarrow \s^2 - R(f)$ lifts uniquely to a map $S_{i}:\Sigma(w_{j}, w_{j+1})\longrightarrow X-C(f)$ over the component of $X - C(f)$ that contains $x_i$, giving therefore a section to $f$ over $\Sigma(w_{j}, w_{j+1}) - \{\}$ of ${\Gamma}$. Thus, being $S_{i}$ homeomorphisms over its image, then $f^{-1}(\Sigma(w_j , w_{j+1}))$ is a collection of Jordan arcs connecting the preimages of $w_j$ to preimages of $w_{j+1}$ by $f$.

Similarly we can prove that the components of $f^{-1}(\s^2 - \Sigma)$  are Jordan domains as it follows.

By \emph{Jordan-Sch\" oenflies theorem}, $\s^{2}-\Sigma$ is a disjoint union of two Jordan domains, say $A$ and $B$. Let $a\in A$ and $b\in B$. 

$f|_{X-\{f^{-1}(\Sigma)\}}:X-\{f^{-1}(\Sigma)\}\longrightarrow \s^{2}-\Sigma$ is a cover. 
Then the fibers of $f$ above $a$ and $b$ contains, each one, $d$ distincts points. For each point $x_i \in f^{-1}(a)$ and $y_i \in f^{-1}(b)$ the inclusion maps $Id_A : A\hookrightarrow \s^2 -\Sigma$ and  $Id_B : B\hookrightarrow \s^2 -\Sigma$ lifts uniquely to a map $S_{1a} : A\longrightarrow X-\Gamma$ and $S_{1b} : B\longrightarrow X -\Gamma$ over the component of $X -\Gamma$ that contains $x_1$ and $y_1$, respectively, giving therefore a section to $f$ over each face of ${\Gamma}$. Thus, being $S_{1a}$ and $S_{1b}$ homeomorphisms over its image, $X-\{f^{-1}(\Sigma)\}$ is a union of $2d$ open sets that are homeomorphic to Jordan domains. Then, we are done. 

The rest of the content was clarified above the proposition.
\end{proof}

\begin{defn}\label{corner-saddle}
A vertex in $C_f \subset V({{\Gamma}})$ will be called a \emph{corner} of $\Gamma$ and a path in ${\Gamma}$ connecting two corners will be called a \emph{saddle-connection} of $\Gamma$.
\end{defn}

Thereby, the realization problem $\ref{realiz-q}$ may now be expressed as:
\begin{quest}\label{realiz-q2}
\textcolor{uibred}{What sort of cell graphs can be realized as a pullback graph, that is, who of them are realized as the embedded graph $f^{-1}(\Sigma)\subset S_g$ for some ramified covering $f:S_g\rightarrow \s^2$ and postcritical curve $\Sigma$?}\end{quest}

This question is motivated by the following 
issue raised by Thurston:
 
\begin{quest}\label{Q2}
\begin{center}
\textcolor{uibred}{What is the shape of a rational map?} \emph{(}see \cite{BillQ}\emph{)}
\end{center}
\end{quest}

According to Proposition $\ref{prop-shape-graph}$, the embedded graphs wondered in \textcolor{uibred}{\textrm{Q}} $\ref{realiz-q2}$must be among those cellular graphs that allow a textcolor{deeppink}{A}-\textcolor{blue}{B} alternating face coloring made up of an equal number of textcolor{deeppink}{A}  and \textcolor{blue}{B} faces, with corners incident to the same face only once. Moreover, for which it is possible to group the vertices suitably that are compatible with a branched covering passport, as shown in figure $\ref{pullb-model1}$. 
This latter condition will be properly presented and examined in the next section $\ref{admg's}$. 

\begin{figure}[H]
\centering
\subfloat{\tikz[remember
picture]{\node(1AL){\includegraphics[width=4cm]{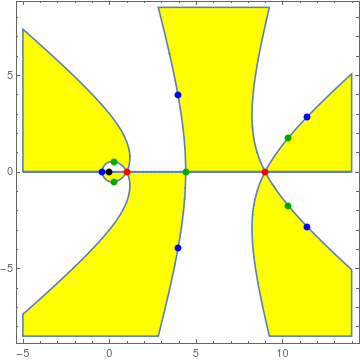}};}}%
\hspace*{1.5cm}%
\subfloat{\tikz[remember picture]{\node(1AR){\includegraphics[width=4cm]{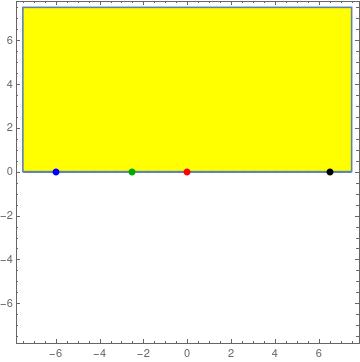}};}}
\caption{}\label{pullb-model1}
\end{figure}
\tikz[overlay,remember picture]{\draw[-latex,thick] (1AL) -- (1AL-|1AR.west)
node[midway,below,text width=1.5cm]{\hspace{.65cm}$f$};}

Note that from the Proposition $\ref{prop-shape-graph}$, \emph{pullback graphs} does not have \emph{saddle-connection} connecting a corner to itself. The cell graphs considered here will not possesses vertices that are incident more than once to a face of the graph. Hence, cell graphs having loops are not considered. Actually, that is a consequence of the \emph{balance conditions}.

In the sequel, we introduce the  \emph{global balance condition}.

\begin{defn}[globally balanced graph]\label{fefn-bg}
A \emph{Globally Balanced Graph} of type $(g, d, m)$
is a cell 
graph on an oriented compact surface of genus $g$, $S_g$, with $2d$ faces, $m$ corners 
 and which admits an \textcolor{deeppink}{A}-\textcolor{blue}{B} alternating face coloring 
 with 
 $d$ faces colored by each color. We say also that such an embedded graph satisfy the \emph{Global balance condition}.
\end{defn}
\begin{figure}[H]
\begin{center}
\subfloat[Globally Blalanced graph of type $(2, 4, 6)$]
{{\includegraphics[width=5cm,height=3.15cm]{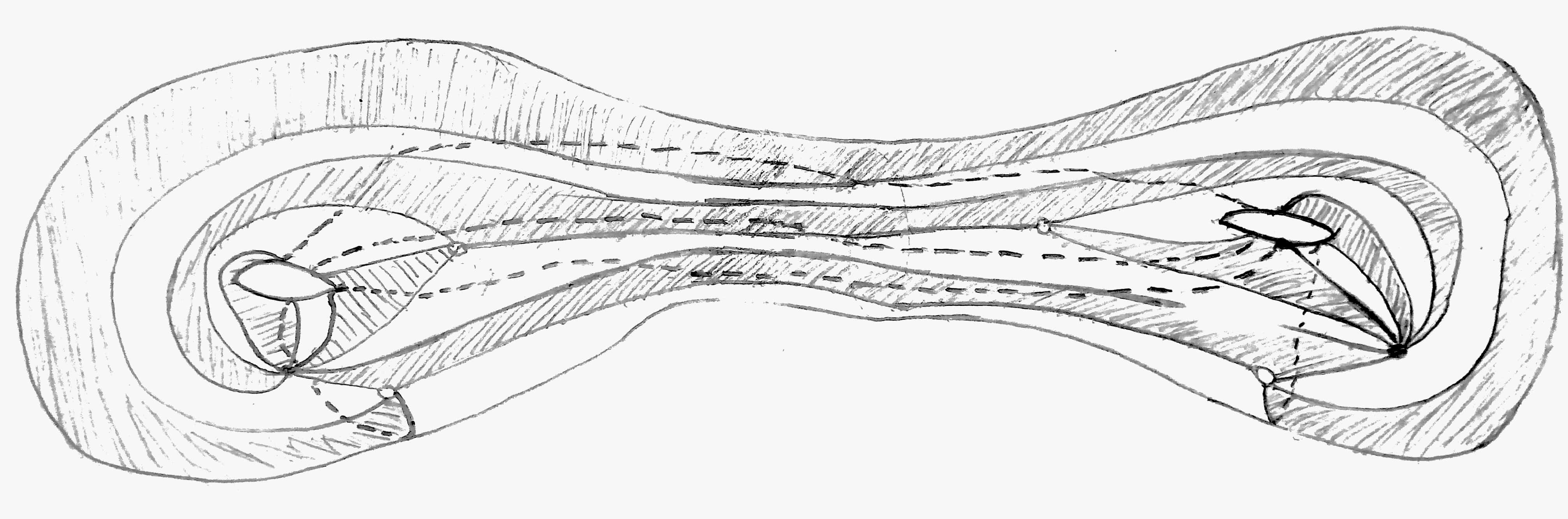}}}\hspace{1cm}
\subfloat[even graph non-globally balanced]
{{\includegraphics[width=4cm]{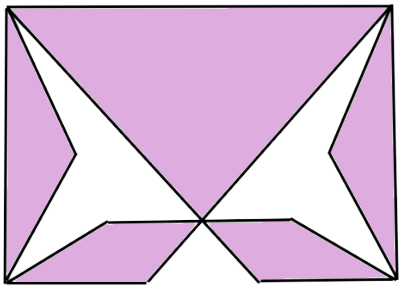}}}
    
    \vspace{-.25cm}
\caption{\hfill}
\label{par_n_gbal}
\end{center}
\end{figure}

Now observe that every graph that admits alternating coloring for its faces has even degrees for all of its vertices, i.e., it is an even graph (there are only two colorings). Yet, as we can see in the Figure $\ref{par_n_gbal}$, an even graphs are not necessarily globally balanced. 


\begin{defn}\label{deg-bg}
The degree of a  globally balanced graph is half of the number of its faces
 \emph{(}i.e., is the number of faces of the same color\emph{)}.
\end{defn}

\begin{lem}\label{mcornern}
The maximal number of corners on a globally balanced graph of degree $d$ and genus $g$ is $2(g+d - 1)$. 
\end{lem}
\begin{proof}
From the \emph{Euler formula},\[2-2g=V(\Gamma)-E(\Gamma)+2d\]
Sice each corner has degree greater or equal to $4$, then $E(\Gamma)\geq\frac{4V(\Gamma)}{2}={2V(\Gamma)}.$

Therefore,
\[V(\Gamma)=2V(\Gamma)-V(\Gamma)\leq E(\Gamma)-V(\Gamma)=2g+2d-2\]
\end{proof}

 
\subsection{Construction of branched coverings from diagrams}\label{admg's}
Now we are going to introduce a class of embedded graphics that encodes a recipe to build a branched covering from them.
 
 \begin{defn}[enriched cell graph]\label{def:enrich-g}
 A \emph{enriched cell graph} is a cell graph with some vertices of valence $2$ \emph{(}we  can think of them as marked points in the graph\emph{)}.
\end{defn} 

\begin{defn}[vertex labeling]\label{def:labeling}
For a graph $\Gamma$, a surjective map $L:V({{\Gamma}}) \longrightarrow J$ from the vertex set to a finite set $J$ is called a vertex labeling of $\Gamma$ by $J$. For a vertex $v\in V(\Gamma)$ such that $L(v)=j$ we write $v_j .$
\end{defn} 

\begin{defn}[admissible vertex labeling]\label{adm-v-l}
Let ${\Gamma}$ be a degree $d>0$ enriched globally balanced graph 
with the same number $m\geq 2$ of vertices incident to each one of its faces \emph{(}here we are also considering vertices of valence $2$\emph{)}.  
A vertex labeling of $\Gamma$ by the cyclically ordered set $\{1<2<\cdots <m\}$ 
is called \emph{admissible labeling} if:
\begin{itemize}
\item[(1)]{The labels $1<2<\cdots <m$ appear cyclically around each face of $\Gamma$. Further, when a face boundary cycle is traversed in the labels' increasing order the preferred color is kept on the left;}
\item[(2)]{For each label $j\in\{1<2<\cdots <m\}$ it holds
\begin{eqnarray}
\displaystyle{\sum_{k=1}^{l_j} \dfrac{deg(v_{j}^{k})}{2}=d}
\end{eqnarray}where those $v_{j}^{k}$'s are the vertices of $\Gamma$ labelled with $j\in\{1<2<\cdots <m\}$, i.e., $\{v_{j}^{1}, v_{j}^{2},\cdots ,v_{j}^{l_j}\}=L^{-1}(j)$.}
\end{itemize}
\end{defn}

\begin{defn}[Admissible Graph]\label{adm-g}
An \emph{admissible graph} is a enriched globally balanced graph with an admissible labeling. The type of an \emph{admissible graph} is its type as a globally balanced graph. 

An \emph{admissible Graph} is \emph{generic} when its corners are all of degree $4$ and have pair-wise distinct labeling.
\end{defn}


As shown in the illustration below, we can have different admissible graphs but with the same underlying globally balanced graph.

\begin{figure}[H]
    \centering
    \subfloat[]
    {{\includegraphics[width=3.3cm,height=2cm]{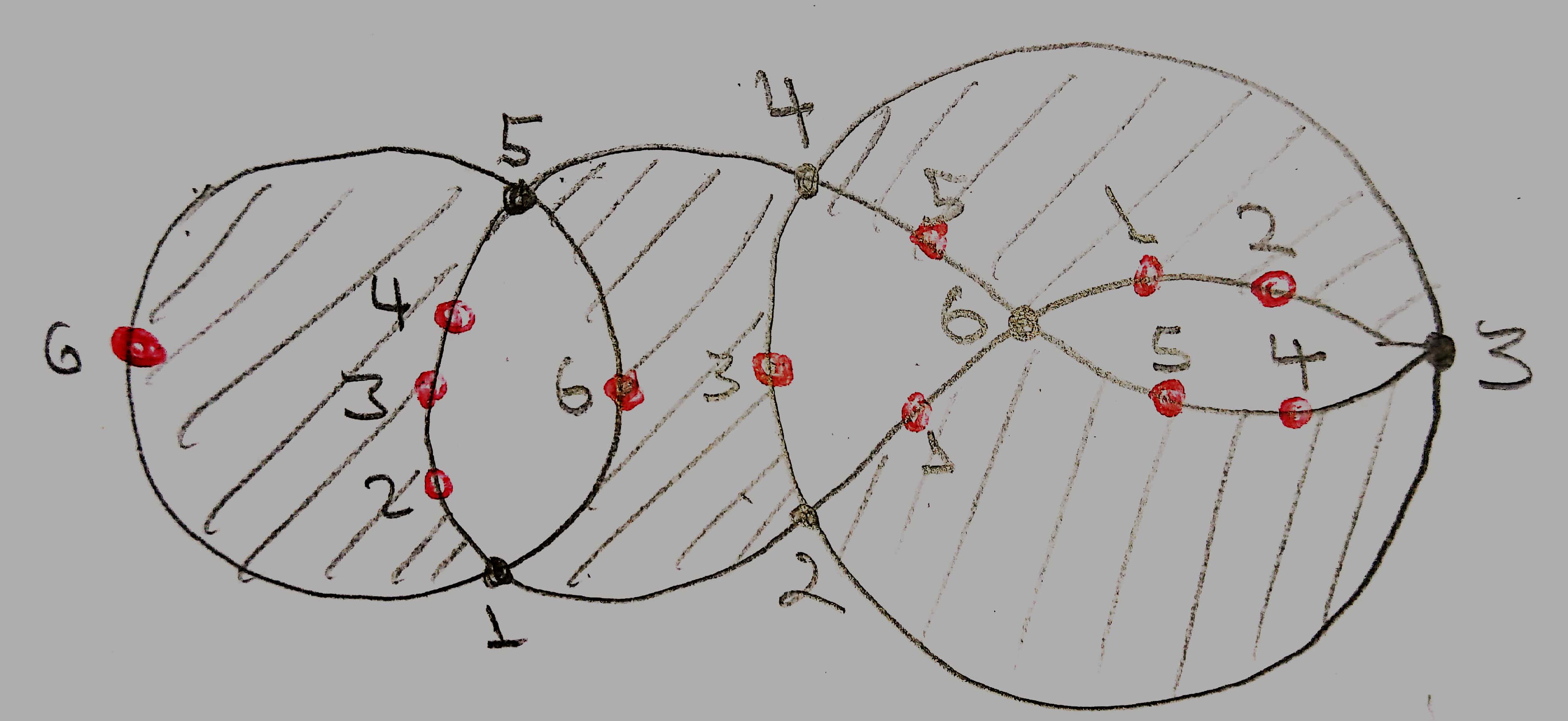}}}
	\quad
    \subfloat[]
    {{\includegraphics[width=3.4cm,height=2cm]{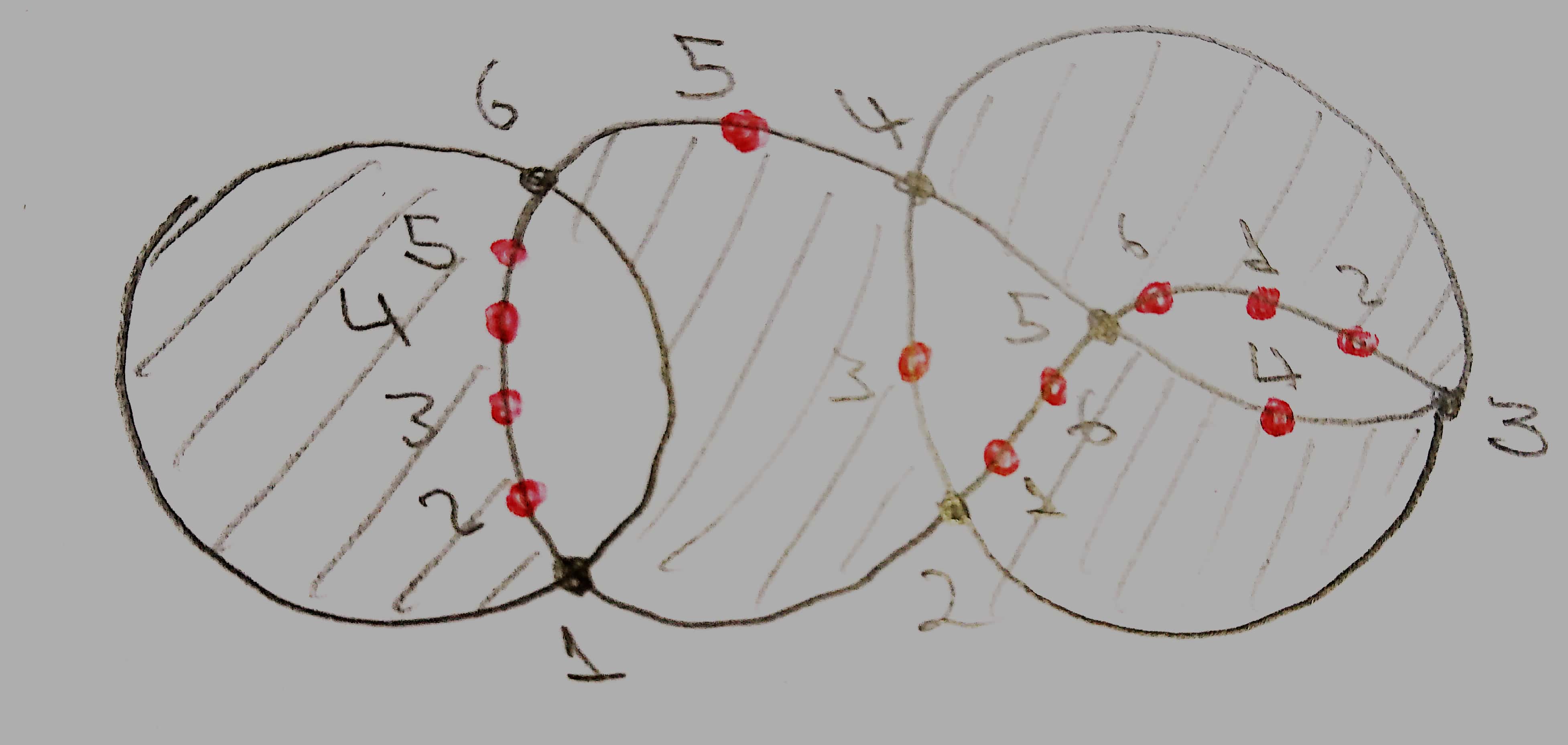}}}
    \quad
    \subfloat[]
    {{\includegraphics[width=2.8cm,height=2cm]{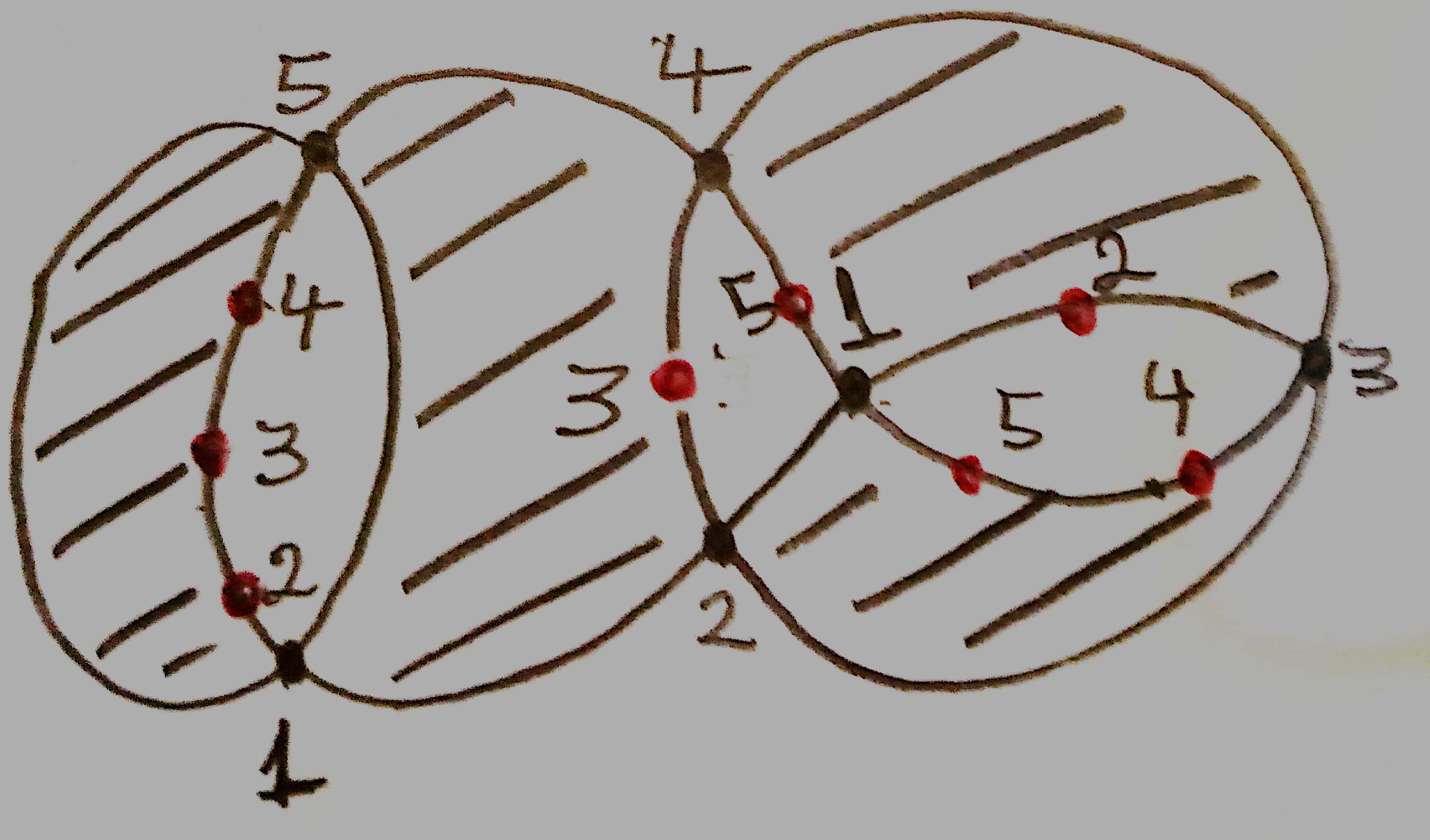}}}
        
    \vspace{-.25cm}
\caption{different admissible graphs from the same planar cell graph}
    \label{fig:h12pullb1}

\end{figure}

\subsubsection{Constructing a branched selfcovering of $\s^2$ from an admissible graph}\label{bcconstruction} 

To each face of $\Gamma$, resorting to Schoenflies’ Theorem, we attach an embedding in $\s^2$ of its closure, such that:
\begin{itemize}\item[]{\paragraph{Gluing conditions:}\begin{enumerate}\label{constr}
\item[(g.1)]{the image of the closure of each \textcolor{deeppink}{A} faces are equal to a closed $2$-cell, say $\Omega\subset\s^2$, and the image of the closure of each \textcolor{blue}{B} faces are equal to $\s^2-int(\Omega)$, where $int(\Omega)$ is the topological interior of $\Omega\subset\s^2$;}
\item[(g.2)]{vertices with the same label have the same image through the face embeddings;}
\item[(g.3)]{the embeddings of two faces with common saddle-connections are equal over those common saddle-connections.}
\end{enumerate}}\end{itemize}

Thus, we have constructed a finite degree continuous map $f$ by gluing together all those embeddings. Since every point in $\s^2$ has exactly d points above it, except those where corners are projected. These points are the branch points.

Let $\Sigma:=\partial \Omega$. 

By construction $f:|\Gamma|-\Gamma\longrightarrow\s^2-\Sigma$ is a local homeomorphism. Let $C\subset\s^2$ be the set of vertices of $\Gamma$ with degree strictly greater than $2$ and  $R:=f(C)\subset\s^2$. Due to $f^{-1}(\Sigma)=\Gamma$ and the coincidence of the embeddings over the saddle-connections, $f$ is a local homeomorphism in each point in $\Gamma - C$. 

Note also that the local degree of $f$ around each point $v^{k}_{j}$ in the fiber of $f$ over the critical value $w_j:=f(\{v_{j}^{1},\cdots ,v_{j}^{l_j}\})$ is equal to $\dfrac{deg(v_{j}^{k})}{2}$. Since each point in any punctured vicinity of $w_j$ possesses exactly that number of points above it around $v^{k}_{j}$ (compare with the figure $\ref{cbc1}$).
 
Therefore, $f$ is a branched covering with passport $\pi=(\pi_1,\cdots ,\pi_m)$ where 
\[\pi_j=\left(\dfrac{deg(v_{j}^{1})}{2}, \dfrac{deg(v_{j}^{2})}{2}, \cdots , \dfrac{deg(v_{j}^{l_j})}{2}\right)\quad\mbox{for each $j\in\{1, 2, \cdots, m\}.$}\]

\begin{figure}[H]
\begin{center}
\includegraphics[width=4cm]{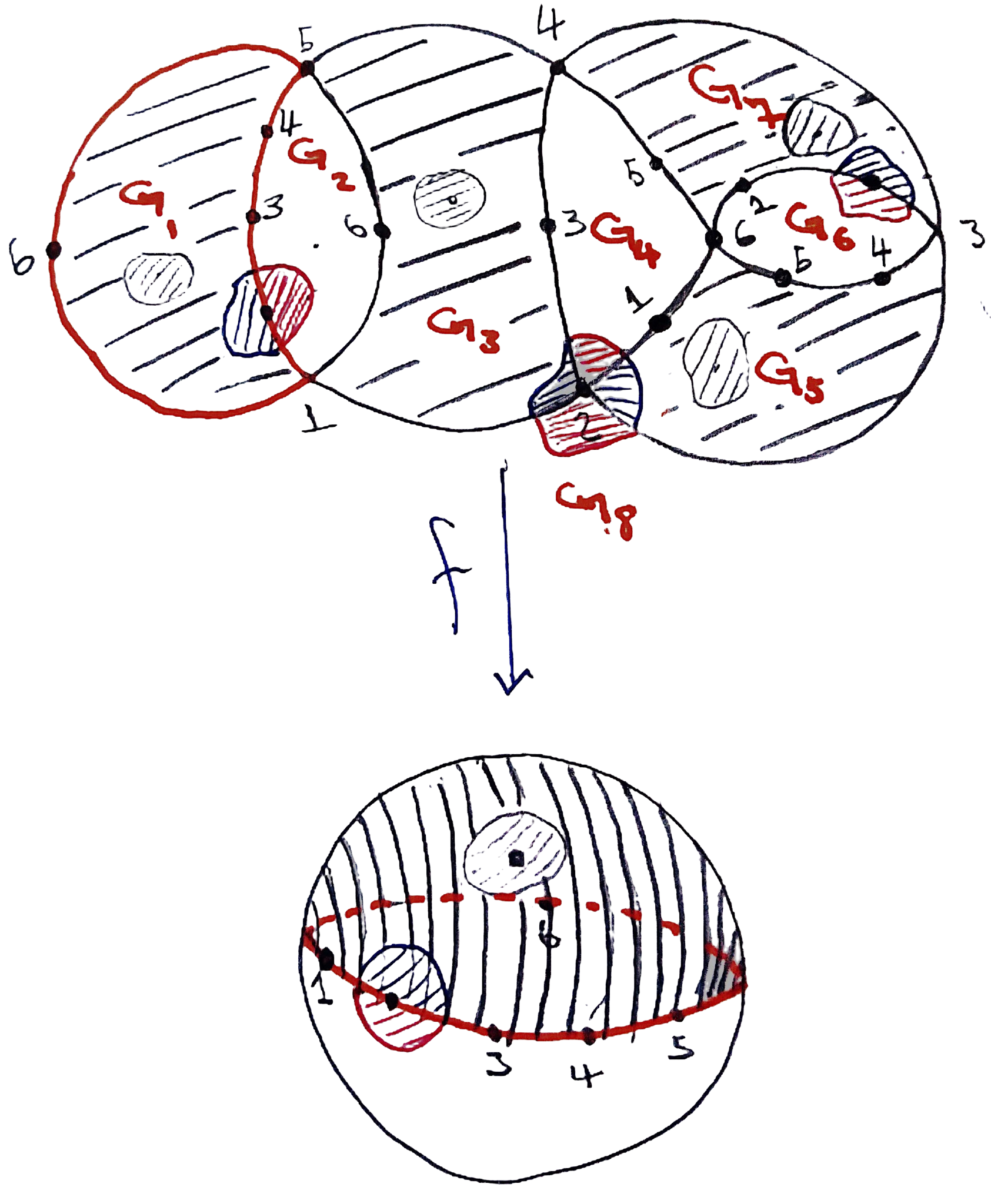}\
\caption{construction of a branched cover $\s^2\rightarrow\s^2$}
\label{cbc1}
\end{center}
\end{figure}

\begin{defn}
Given an admissible graphs $\Gamma$ a system of face embeddings for $\Gamma$ is a attachment to each face of $\Gamma$ of a embedding of its closure under the gluing conditions $\ref{constr}$. Given a enumeration $\{F_i\}_{i=1}^{i=2d}$ of $F(\Gamma)$ we denote a system of face embeddings for $\Gamma$ as $\amalg g_i$ where $g_i :\myov{F_i}\rightarrow \s^2$ is the embedding attached to $F_i$.
\end{defn} 

Hence, we have obtained the following theorem.

\begin{thm}\label{bcfromg02}
For each genus $g$ admissible graph with a system of face embedding $(\Gamma, \amalg f_i )$  there exist a branched covering $S_g\rightarrow{\s^2}$ for which $\Gamma$ is a pullback graph.
\end{thm}
{Now, by the \emph{uniformization theorem} there exist a unique homeomorfism \[\mu:(\s^2, f(C)) \rightarrow(\CPu, \{0,1,\infty,v_1,\cdots , v_{m-3}\})\] for wich $\mu (\alpha)=1$, $\mu (\beta)=0$ and $\mu (\gamma)=\infty$ defining a complex structure over $(\s^2, f(C))$ for a choice $\{\alpha, \beta, \gamma\}\subset f(C)$. Since $f$ is a local homeomorphism on $\s^2 - C$, that complex structure can be pulled back by $f$ to a complex structure on $(|\Gamma|, C)$, say $\nu:=f^{\ast}\mu : (|\Gamma|, C)\rightarrow(X,\{c_1, c_2, \cdots, c_{k}\})$, for where $k=|V(\Gamma)| = C$. }

 Therefore, the map 
 \[F = F_{\nu\,\mu}:=\mu\circ f\circ\nu^{-1}:(X,\{c_1, c_2, \cdots, c_{2d-2}\})\rightarrow(\CPu,\{0,1,\infty, v_1,\cdots, v_{k}\})\] is a holomorphic function. 

Then, we have shown the following.

\begin{cor}\label{bcfromg2}
For each genus $g$ admissible graph with a system of face embeddings $(\Gamma, \amalg f_i )$ there exist a holomorphic branched covering $S_g\rightarrow{\CPu}$ \emph{(i.e.}, a meromorphic function\emph{)}, having $\Gamma$ as a pullback graph.
\end{cor}

Notice that each system of face embeddings for an admissible graph gives a branched covering. Thus, is there some relation between those maps?

\begin{prop}
Let $\mathcal{G}$ and $\mathcal{H}$ be two equivalent admissible graphs and let $\amalg g_i$ and $\amalg h_i$ be two system of face embeddings for them  respectively. Then, there are orientation preserving homeomorphisms 
$\phi:\s^2\rightarrow\s^2$ and $\psi:|\mathcal{G}|\rightarrow|\mathcal{H}|$ such that 

\begin{eqnarray}
\phi\circ{g}=h\circ \psi
\end{eqnarray} 
where $g$ and $h$ are the branched coverings 
contructed  from $(\mathcal{G}, \amalg g_i)$ and $(\mathcal{H}, \amalg h_i)$, respectively.
\end{prop} 
\begin{proof}
Let $\Sigma_g := g(\mathcal{G})$ and $\Sigma_h := h(\mathcal{H})$. From \emph{Schöenflies} theorem we get a homeomorphism $\phi:\s^2\rightarrow\s^2$ such that $\phi(\Sigma_g)=\Sigma_h$ and if $V(\mathcal{G})\subset \s^2-\Sigma_g$ and $V(\mathcal{H})\subset \s^2-\Sigma_h$ are the cell image of the \textcolor{deeppink}{A} faces of $\mathcal{G}$ by $G$ and $H$, then $\phi(V_G)=V_H$ (with the same for the images by $G$ and $H$ of the \textcolor{blue}{B} faces). Furthermore, $\phi \circ G|_{V(\mathcal{G})} = H \circ I|_{V(\mathcal{G})}$.

Therefore, the map $\psi(z):=H^{-1}|_{I(\myov{F})} \circ \phi \circ G:|\mathcal{G}|\rightarrow|\mathcal{G}|$ for each $z\in \overline{F}$ for each face $F$ of $\mathcal{G}$, determines a homeomorphism from $|\mathcal{G}|$ to $|\mathcal{H}|$ that satisfies
\begin{eqnarray}
\phi\circ{G}=H \circ \psi
\end{eqnarray}
\end{proof}

\begin{lem}\label{prop2}
Let $\mathcal{G}$ and $\mathcal{H}$ two equivalent cell graphs. 
Let $g:|\mathcal{G}|\rightarrow \s^2$ and $h:|\mathcal{H}|\rightarrow \s^2$ two continuous surjective maps that restricts to homeomorphisms over the topological closure of each face and such that 
\begin{eqnarray}
g(\overline{F})=h(I(\overline{F}))
\end{eqnarray}
for each face $F$ of $\mathcal{G}$, for $I$ being a homeomorphism atesting the equivalence of the cell graphs. 
Then there exist a homeomorphism $\mathcal{I}:|\mathcal{G}|\rightarrow |\mathcal{H}|$ (actually, $|\mathcal{G}| = |\mathcal{H}|$) such that
\begin{eqnarray}\label{tp2}
g=h \circ{} \mathcal{I}. 
\end{eqnarray}
\end{lem}
\begin{proof}
Via the homeomorphism $I$ and the property $(\ref{prop2})$ define
\begin{eqnarray}
\mathcal{I}(z):= h^{-1}|_{h(I(\overline{F}))}\circ{}f(z)
\end{eqnarray} 

for each $z\in \overline{F}$ for each face $F$ of $\mathcal{G}$. 

Thus, $\mathcal{I}$ is a homeomorphism due the hypotesis that $f$ and $h$ restricts to homeomorphisms over each closed face of $\mathcal{G}$ and $\mathcal{H}$ and by construction it satisfies  $(\ref{tp2})$.
\end{proof}

\begin{prop}
Given two equivalent admissible graphs, say $\mathcal{G}$ and $\mathcal{H}$, if there exist an orientation preserving homeomorpism $\varphi:\s^2\rightarrow\s^2$ such that 

\begin{eqnarray}
h(I(\overline{F}))=\phi(g(\overline{F}))
\end{eqnarray} 

for each face of $\mathcal{G}$, then the meromorphic functions produced from it as in the preceding construction $\ref{bcconstruction}$ are equivalents if the face embeddings are isotopic relative to the critical value set.
\end{prop}

\begin{proof}

First, what we mean by saying that the face embeddings are isotopic relative the critical value set is that the two Jordan curves image of the boundary of some face (therefore, of any one) of each graph from the face  embeddings are isotopic relative to the critical value set. 


The isotopy hypothesis guarantees the existence of a homeomorphism $\phi:\s^2\rightarrow\s^2$ compatible with the face embeddings, i.e.,

\begin{eqnarray}
h(I(\overline{F}))=\phi(g(\overline{F}))
\end{eqnarray} for each face of $\mathcal{G}$.

So, Proposition $\ref{prop2}$ gives a homeomorphism $\mathfrak{I}$ such that
\begin{eqnarray}
h\circ{}\mathfrak{I}=\phi\circ{}g
\end{eqnarray}

Let $G:=\mu_g\circ{}g\circ{}\nu_{g}^{-1}$ and $H:=\mu_h\circ{}h\circ{}\nu_{h}^{-1}$ be those two rational functions as anounced, where $\nu_g, \mu_g, \nu_h , \mu_h$ the uniformizing maps of the domain and codomain of the topological branched coverings $g$ and $h$ constructed from $\mathcal{G}$ and $\mathcal{H}$ (as in $\ref{bcconstruction}$). 

Since $\mathfrak{I}$ and $\phi$ fix the distinguished corners $\alpha$,$\beta$, and $\gamma$ (the normalization), and $\nu_g(0)=\nu_h(0)=\alpha$,$\nu_g(1)=\nu_h(1)=\beta$, $\nu_g(\infty)=\nu_h(\infty)=\gamma$, $\mu_g(\alpha)=\mu_h(\alpha)=0$, $\mu_g(\beta)=\mu_h(\beta)=1$ and $\mu_g(\gamma)=\mu_h(\gamma)=\infty$, follows that $\nu_h\circ{}\mathfrak{I}\circ{}\nu_g^{-1}=Id_{\CC}$ and $\mu_h\circ{}\phi\circ{}\mu_g^{-1}=Id_{\CC}=Id_{\CC}$ as they are conformal automorphisms of $\CC$ that fixes three points.

Therefore,
\begin{eqnarray}
H &=& \mu_{h}\circ{}h\circ{}\nu_{h}^{-1}\\ \nonumber
&=& \mu_{h}\circ{}h\circ{}\mathfrak{I}\circ{}\nu_{g}^{-1}\\ \nonumber
&=& \mu_{h}\circ{}\phi\circ{}g \circ{}\nu_{g}^{-1}\\ \nonumber
&=& \mu_{g}\circ{}g\circ \nu_{g}^{-1}\\ \nonumber
&=& G\\ \nonumber
\end{eqnarray}

\begin{center}
\begin{tikzcd}[row sep=large, column sep = large]
\cc \arrow{r}{\nu_g}\arrow{d}{Id_{\cc}} & (\s^2 ,\mathcal{G})  \arrow[dashrightarrow, xshift=-1.4ex]{d}{\mathfrak{I}}  \arrow{r}{g} \arrow[xshift=1.4ex]{d}{I} &(\s^2 ,g(C))\arrow{r}{\mu_g} \arrow{d}{\phi} &\cc \arrow{d}{Id_{\cc}}  \\
\cc \arrow{r}{\nu_h} & (\s^2 ,\mathcal{H})\arrow{r}{h}&(\s^2 ,h(C))\arrow{r}{\mu_h}&\cc
\end{tikzcd}
\end{center}

\end{proof}


\subsection{A special case: real rational functions from diagrams}\label{realcase?}

Now, let's look at a special class of admissible graphs.

In the next section, we will achieve a full generalization of a theorem by Thurston proved firstly for generic branched self-coverings of the $2$-sphere.
\section{General version of a theorem by Thurston}\label{sect-thurstonthm}



An important featuring underlying the definition of the Thurston's Local Balance Condition is that on the 2-sphere every simple closed curve divide the 2-sphere in two 2-cells. That is the content of the \emph{Jordan Curve Theorem}. That fact does not hold for positive genus surfaces. Therefore, if we want to use Thurston’s concept more broadly, we must address this flaw.  


\begin{defn}
A \emph{simple closed curve} $\gamma$ into a surface $S$ 
 is \emph{separating} if $S-\gamma$ has two components. Otherwise, $\gamma$ is \emph{non separating}.
\end{defn}

\begin{defn}[cobordant multcycle]\label{sub-surfs}
Let $\Gamma$ be a cell graph 
that admits an \textcolor{deeppink}{A}-\textcolor{blue}{B} alternating face coloring. We say that a collection of disjoint cycles $L:=\{\gamma_1, \cdots, \gamma_k\}$ of $\Gamma$ are a \emph{cobordant multicycle of} $\Gamma$ if:
\begin{itemize}
\item[i.]{$S-\{\gamma_1, \cdots , \gamma_k\}$ is disconnected;}
\item[ii.]{there is a connected component $R$ of $S-\{\gamma_1, \cdots , \gamma_k\}$ such that $\partial R = \bigsqcup_{n=1}^{k}\gamma_n$.}
\end{itemize}
We call $L$ by \emph{positive cobordant multicycle} of $\Gamma$ when each cycle $\gamma_n \in L$ is positive. And, R is the interior of L.
\end{defn} 

Now we are on time to present the \emph{Local Balance Condition}.


\begin{defn}[Locally Balanced Graph]\label{localbalance}
Let $\Gamma$ cell graph.
 with an alternating \textcolor{deeppink}{A}-\textcolor{blue}{B} face coloring.
We say that $\Gamma$ is \emph{locally balanced} if for any \textcolor{deeppink}{A}-\textcolor{blue}{B} alternating face coloring and for any \emph{positive cobordant multicycle of} $\Gamma$ the number of $\textcolor{deeppink}{A}$ faces inside it (i.e, on the interior of that multicycle) is strictly greater than the number of $\textcolor{blue}{B}$ faces.
\end{defn}

Thurston established the following definition for $4$-regular planar cell graphs. 

\begin{defn}[local balance condition from Thurston \cite{STL:15}]\label{thurstonlb}
A $4$-regular planar cell graphs $\Gamma$ is \emph{locally balanced} if for any \textcolor{deeppink}{A}-\textcolor{blue}{B} alternating face coloring and for every positive cycle of $\Gamma$ the number of $\textcolor{deeppink}{A}$ faces inside it, is strictly greater than the number of $\textcolor{blue}{B}$ faces there.
\end{defn}


Although in the planar situation Definition $\ref{localbalance}$ it seems to be more restrictive than the one given by Thurston, they are actualy equivalent.


\begin{prop}[meaningfulness of Definiton $\ref{localbalance}$]
For \emph{planar} globally balanced graphs those two definitions of local balancedness are equivalents.
\end{prop}
\begin{proof}
Thanks to the \emph{Jordan Curve Theorem} is immediate that Definition $\ref{localbalance}$ implies the Definition $\ref{thurstonlb}$.

So, let's prove the reverse implication. That is, we will guarantee that if a planar balanced graph that satifies the Definition $\ref{thurstonlb}$ then it also enjoys the Definition $\ref{localbalance}$. 

Let $\Gamma$ be a planar globally balanced graph with an alternating \textcolor{deeppink}{A}-\textcolor{blue}{B} face coloring and $\Lambda$ be a \emph{cobordant positive multicycle of} $\Gamma$ with interior $R$.

Let $Y$ be a connected component of $\s^2 -R$. Since $R$ is connected the boundary of $Y$ has only one component $\gamma\in L$. 

Thus $\gamma$ encloses the complement of $Y$ leaving \textcolor{deeppink}{A} faces on its left side.

Hence, from the \emph{local balance} condition we conclude that are more pink faces than blue ones outside $Y$.

Let $Y_1 , Y_2, \cdots, Y_n$ be the components of $\s^2 - R$, and $\textcolor{deeppink}{a_{k}}$ and $\textcolor{blue}{b_{k}}$ the number of \textcolor{deeppink}{A} faces into $Y_k$ and the number of \textcolor{blue}{B} faces into $Y_k$, respectively. $\textcolor{deeppink}{a_{R}}$ and $\textcolor{blue}{b_{R}}$ are the numbers of \textcolor{deeppink}{A} faces and \textcolor{blue}{B} faces into $R$.

Hence, from the above argumentation
\begin{eqnarray}
\textcolor{deeppink}{a_{k}}<\textcolor{blue}{b_{k}}
\end{eqnarray}for each $k=1,2, \cdots , n$

And, since,
\begin{eqnarray}
\textcolor{deeppink}{a_R} + \sum_{k=1}^{n}\textcolor{deeppink}{a_{k}}=\textcolor{blue}{b_R} + \sum_{k=1}^{n}\textcolor{blue}{b_{k}}
\end{eqnarray}

Then,

\begin{eqnarray}
\textcolor{deeppink}{a_R} >\textcolor{blue}{b_R}
\end{eqnarray}

\end{proof}

\begin{defn}[Balanced Graph]\label{balance}
A \emph{balanced graph} is a cell graph on an oriented compact surface that is both globally and locally balanced. The \emph{type of a balanced} graph is its type as a globally balanced graph.
\end{defn}
\begin{figure}[H]
\begin{center}
\subfloat[balanced graph of type (0,4,6)]
{{\includegraphics[width=4cm,height=4cm]{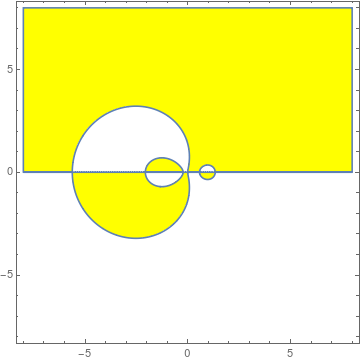}}}\qquad
\subfloat[balanced graph of type (1,4,4)]
{{\includegraphics[width=4cm,height=4cm]{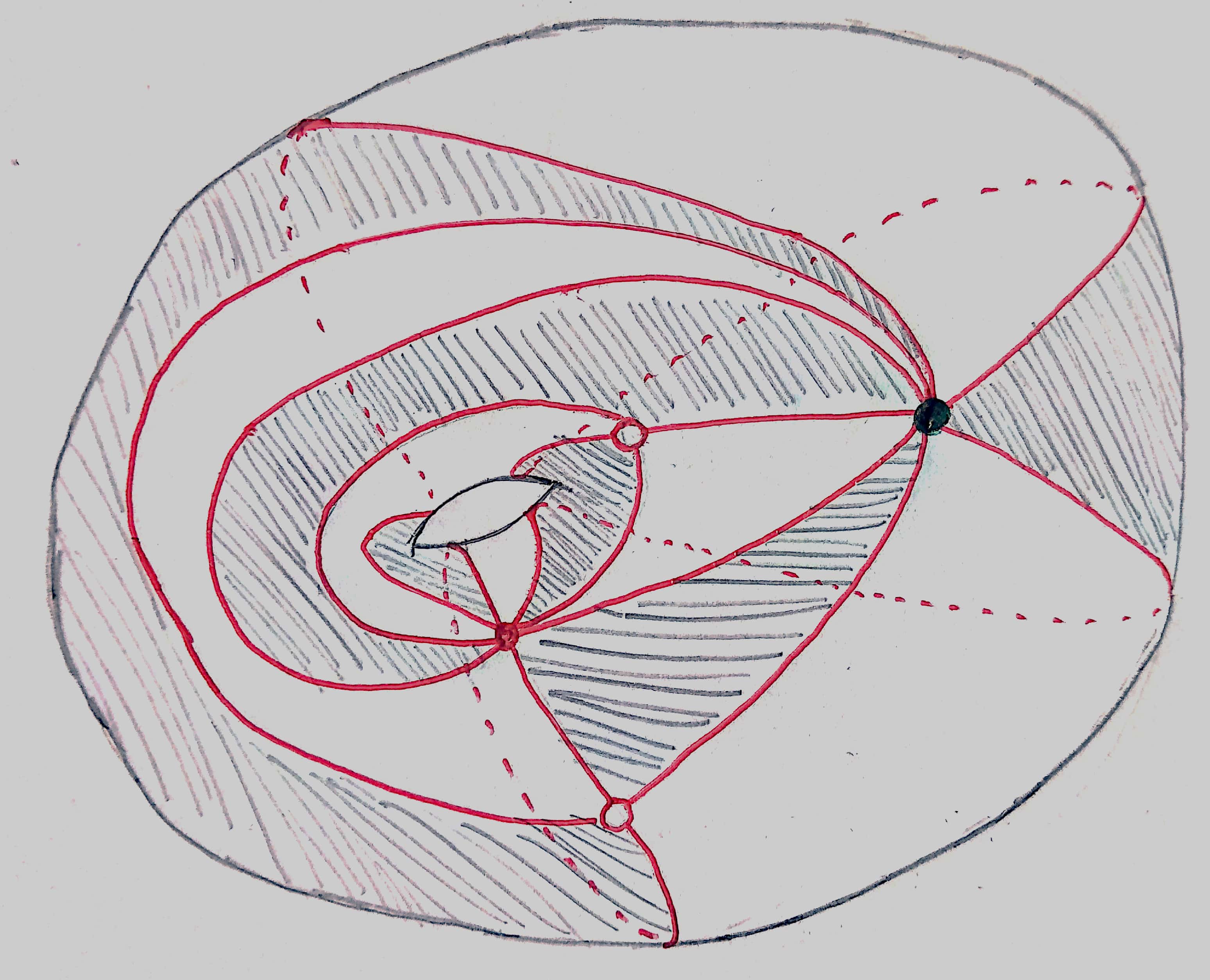}}}
\label{bg-000}
\end{center}
\end{figure}

\begin{mnthm}{A}[\textcolor{uibred}{General version of a theorem by Thurston}]\label{teo-a}
An cellularly embedded graph $\Gamma$ into a genus $g$ oriented compact surface $S_g$ is a pullback graph if and only if it is a balanced graph.   
\end{mnthm}

\begin{proof}
We will follow closely the initial proof given by Thurston\cite{STL:15}.


The crucial insight 
for to promote a pullback graph to a balanced graph consists of recognizing its combinatorial structure as stemming from a perfect matching on an adjacent bipartite graph.

Thus, let $f:X\rightarrow\s^2$ be a degree $d$ branched cover with $m$ critical values. Consider $\Sigma$, a post-critical curve for $f$, and let ${\Gamma}$ be the corresponding pullback graph. 
Let $A\subset \s^2$ and $B\subset \s^2$  be the two connected components of $\s^2 - \Sigma$.

For each $2$-valent vertex we mark a dot into those two face of ${\Gamma}$ incident to it. 
 Hence, each $2$-valent vertex will have two marked dots corresponding to it into each one of its two neighboring faces. 
Since the boundary of each face contains exactly $m$ vertices on its boundary, after we did that, each face of $\Gamma$ will contain $m-e$ dots, where $e$ is the number of corners around it. 
Each dot corresponds 
to a different critical value of $f$.


Consider from $\Gamma$ the graph $G:=G(\Gamma):=(V=DA\sqcup DB, E)$ where $DX$ is the set of dots from those faces whose image by $f$ is $X\in\{A, B\}$ and $E$ is the adjacency relation of $\Gamma$. 

Then, that spliting procedure of the $2$-valent vertices described above provides a \emph{perfect matching} on the graph $G$.

On the other hand, if we have a balanced graph $\Gamma$ we can also construct that adjacent graph $G$ inserting $m-e_F$ dots into each face $F$ of $\Gamma$ being $m$ the number of corners of $\Gamma$ and $e_f$ the number of corners incident to $F$. The vertex set of $G$ in this case is partitioned into two subsets with respect to the face coloring of the balanced graph $\Gamma$.

Thus, to cluster together pairs of these points from adjacent faces over its neighboring saddle-connections then becomes a graph theoretical \emph{matching problem}. 

Now, the existence of a \emph{perfect matching} on $G$ will allow us to enrich $\Gamma$  to a new graph by inserting 2-valent vertices for each pair of dots matched. We keep to denote the new graph by $\Gamma$. With the enrichment, $\Gamma$ now has $m$ vertices incident to each face. 

Finally, if we show that $\Gamma$ supports an admissible vertex labeling then from Theorem $\ref{bcfromg02}$ we'll be done.

\paragraph{Let's prove the \underline{\emph{\texttt{if}}} part:}
Let $\Gamma\subset X$ be a pullback graph on the compact oriented surface $X$ with post-critical curve $\Sigma$.

From Proposition $\ref{prop-shape-graph}$ follows that the faces are \emph{Jordan} domain's.

Color by \textcolor{deeppink}{pink} the interior of $\Sigma \subset\s^2$ and call it by $P$ and color by \textcolor{blue}{blue} the another component of $\s^2-\Sigma$ and call it by $B$. 

Each point $p\in P$ and $b\in B$ possesses exactly $d$ distinct preimages in $X-\Gamma$, since all critical values are on $\Sigma$. Due to the continuity of $f$ a preimage $\tilde{p}\in f^{-1}(p)$ and $\tilde{b}\in f^{-1}(b)$ can not be in the same face of $\Gamma$, say $F$, for otherwise, we could connect $\tilde{p}$ and $\tilde{b}$ by a curve $\gamma$ into $F$ and in this way $f(\gamma)$ will be a connected set connecting $p\in P$ to $b\in B$ but being interelly contained into $f(F)$ that is equal to $P$ or $B$, what is certantily impossible, since  $P$ and $B$ are disjoint open set. Since $f:X-\Gamma\rightarrow \s^2-\Sigma$ is a local homeomorphism, we also can not have $\tilde{p_0},\tilde{p_1}\in f^{-1}(p)$ into the same face (recall the lifting property of local homeomorphisms). The same, for sure, works for that points over $b$. Therefore, there are $d$ faces of $\Gamma$ colored pink and $d$ faces of $\Gamma$ colored blue. This means that $\Gamma$ is \emph{globally balanced}.


Let $L=\{\gamma_1 , \cdots , \gamma_k\}$ be a cobordant positive multicycle of $\Gamma$ with interior $R$.

Let:
\begin{itemize}
\item[(1)]{$E_{n}>0$ to be the number of corners of $\Gamma$ in $\gamma_n$ that do not are incidente to blue faces inside $\gamma_n$, for each $n\in\{1, \cdots , k\}$;}
\item[(2)]{$a$ be the number of pink faces in $R$;}
\item[(3)]{$b$ be the number of blue faces in $R$;}
\item[(4)]{$D_A$ be the number of dots into those pink faces in $R$;}
\item[(5)]{$D_B$ be the number of dots into those blue faces in $R$;}
\end{itemize}

Then:

\begin{itemize}
\item[(1)]{
since the number of edges $e_j$ bordering a face $B_j$ is equal to the number of corners on its boundary, it follows that
\begin{eqnarray}\label{eqI}
D_B = (m-e_1)+(m-e_2)+\cdots+(m-e_B)=mb-(\sum_{j=1}^{b}e_j)
\end{eqnarray}}
\item[(2)]{and
\begin{eqnarray}\label{eqII} 
D_A = ma-(\sum_{j=1}^{b}e_j)-\sum_{n=1}^{k}E_{n}
\end{eqnarray}}
\end{itemize}

Suppose 
$E_{n}=0$, for each $n\in\{1,\cdots ,k\}$.
Then each connected component of $X-R$ is a \emph{simply connected} domain. This stems from the fact that $\Gamma$ to be conected and $E_n=0$ to imply that each positive cycle $\gamma_{n}$ to be incident to only one blue face outside $R$. Therefore, each component of $X-R$ is a blue face and since $\Gamma$ has so many blue as pink faces, say $d>0$, it follows:

\begin{eqnarray}
b=d-k<d=a
\end{eqnarray}

Now, suppose $E_n >0$ for at least one $n\in\{1\cdots , k\}$.
Then, by the necessary condition from the \emph{marriage theorem} $\ref{ksamento}$ we have:

\begin{align}\nonumber
mb-(\sum_{j=1}^{b}e_j)=D_B &\leq  D_A = ma-(\sum_{j=1}^{b}e_j)-\sum_{n=1}^{k}E_{n} \\ \nonumber 
\shortintertext{since $\sum_{n=1}^{k} E_{n} \geq 1$ ,}
 b&<a 
\end{align}

Thus, $\Gamma$ is \emph{locally balanced}.
\paragraph{Now, let's prove the \underline{\texttt{\emph{only if}}}\, part :}\hspace{6cm}

Let $\Gamma$ be a balanced graph with $m$ corners.

Since each face $F$ of $\Gamma$ is a Jordan domain the number of saddle-connections of $\Gamma$ surrounding $F$ is equal to the number of corners on $\partial{F}\subset\Gamma$. 

Recall that each face $F$ of $\Gamma$ contains $m-e_F$ dots, where $e_F$ is the number of corners incident to $F$.

Let $S$ be an arbitrary set of dots from blue faces of $\Gamma$.

Then the task is: \emph{to show that the set of potential mates for $S$ is at least so large as $S$}. That is the sufficient condition of the \emph{Hall's marriage Theorem} \cite{BondyG}.

Note that the potential mates for a dot into a blue face is exactly the same set of potential mates for any other dot from the same face. Therefore, we can change $S$ adding to it all the remains dots in a face that already has at least one of its dots in $S$. That change will not affect the number of potential mates and, of course, the condition is satisfied for any subset of dots from that enlarged set $S$ whether it itself satisfies the condition. Therefore, due to that, we will take $S$ as being the subset of all dots from a collection $ U $ of blue faces of $\Gamma$.

Denote by $R$ the topological closure of the collection $U$ together with its neighboring pink faces, i.e., $R$ is the union of the faces in $U$ with its neighboring pink faces and all boundaries of those faces.

Then the dots inside pink faces in $R$ are exactly those potential mates for the dots into $S$.

Note that the boundary of $R$ leaves pink faces in its left side, except at the corners.

If the interior of $R$ is not connected, then dots into blue faces of one component can only be matched with those dots inside pink faces from the same connected component of the interior of $R$. Hence we should have enough mates for the individuals of $S$ in each connected component of the interior of $R$. In this way we will have enough mates in $R$ for all individuals. Then is enough to assure the condition for each connected component what allows us to consider $R$ with the interior connected.

Let:
\begin{itemize}
\item[(1)]{$D_A$ denote the number of dots into pink faces inside $R$;}
\item[(2)]{$D_B=|S|$ denote the number of dots into blue faces inside $R$;}
\item[(3)]{$E_1$ be the number of corners on $\partial{R}$ that have only one face from $R$ neighboring it(that number was the number $E_\gamma$ when we prove the local balance condition of a pullback graph above);}
\item[(4)]{$\mu_k$ be the number of corners on $\partial{R}$ that have $k$ blue faces incident to it  from $R$ ;}
\item[(5)]{$\nu_k$ be the number of corners in the interior of $R$ with degree $k$;}
\item[(6)]{$a$ be the number of pink faces in $R$;}
\item[(7)]{$b$ be the number of blue faces in $R$.}
\end{itemize}



Thus, going back to the equations $\ref{eqI}$ and $\ref{eqII}$ we have:
\begin{eqnarray}\nonumber
D_B = mb-\sum_{j=1}^{b} e_j &=& mb-\frac{1}{2}\left(\sum_{j=2}^{d}2\mu_{j} + \sum_{j=2}^{d}{2j}\nu_{2j}\right)\\
&=& mb-\left(\sum_{j=1}^{d}\mu_{j}+\sum_{j=2}^{d}{j}\nu_{2j}\right)
\end{eqnarray}
and
\begin{eqnarray}\nonumber
D_A &=& ma-\frac{1}{2}\left(2E_{1}+\sum_{j=1}^{d}2\mu_{j} + \sum_{j=2}^{d}{2j}\nu_{2j}\right)\\
&=& ma-\left(E_{1}+ \sum_{j=1}^{d}\mu_{j} + \sum_{j=2}^{d}{j}\nu_{2j}\right)
\end{eqnarray} 

From the local balance condition we have $b \leq a-1$, and we also have $E_1 + \sum_{j=1}^{d}\mu_{j} + \sum_{j=2}^{d}\nu_{2j}\leq m$ where $m$ is the total number of corners of $\Gamma$. 

Hence
\begin{eqnarray}
D_A &=& ma-\left(E_{1}+ \sum_{j=1}^{d}\mu_{j} + \sum_{j=2}^{d}{j}\nu_{2j}\right)\\
&\geq& ma-\left(m+ \sum_{j=1}^{d}\mu_{j} +\sum_{j=2}^{d}{j}\nu_{2j}\right)\nonumber\\
&=&m(a-1)-\left(\sum_{j=1}^{d}\mu_{j} + \sum_{j=2}^{d}{j}\nu_{2j}\right)\nonumber\\
&\geq &  mb-\left(\sum_{j=1}^{d}\mu_{j}+\sum_{j=2}^{d}{j}\nu_{2j}\right)\nonumber\\
&=&D_B
\end{eqnarray} 

That is the desired inequality.

Therefore, we have proved that for an arbitrary set $S$ of dots from blue faces of $\Gamma$ the set of potential mates for those dots into $S$ is so large as $S$. Then the \emph{Hall's Marriage Theorem} \cite[Theorem ~2.1.2.]{Diest} with the global balancedness assures the existence of a \emph{perfect matching}.

For each pair of dots matched we get a new vertex on the common side separating the faces containing those dots. These new vertices are taken distinct for each matched pair of dots from the same pair of faces.

Then, $\Gamma$ was enriched into a new graph, now with a bunch of $2$-valent vertices inserted, that we shall continue denoting by $\Gamma$.

But in addition to having $m$ vertices incident to each face, these vertices must be numbered cyclically (regarding the graph orientation) in such a way that the number at a corner given from each face labeling incident to it is the same and, furthermore, with such labeling being in accordance with an admissible passport. With `` to be in accordance with a passport '' we mean that the sum of half the degree of the vertices for a fixed label is equal to the degree $d$ of $ \Gamma $, for each label $ j\in \{1, 2, \cdots , m \}$.

Thus we have to ensure that we can always perform a vertex labeling with that especifications on such a enriched balanced graph. That is, every balanced graph is an admissible graph. Therefore, from Theorem $\ref{bcfromg2}$, we will be done!

\begin{lem}\label{lththm}
The enriched balanced graph obtained above is admissible.
\end{lem}
\begin{proof}[proof of the lemma $\ref{lththm}$]
We must display 
one admissible vertex labeling for $\Gamma$ (the enriched graph). $\Gamma$ has $m$ corners
. We can construct an admissible vertex labeling $N:V(\Gamma)\rightarrow \{1<2<\cdots<m\}$ 
inductively, 
as follows.

First, choose a pink face $F_1 \in F(\Gamma)$ with a numbering of the $m$ vertices incident to it by $1, 2, \cdots, m$ appearing in this order around the face keeping it on the left side. 

For a (labeled) corner adjacent to $F_1$, say $c_1$, we consider all the pink faces incident to it. Then we complete the labeling of the left $m-1$ vertices on each face respecting the already labeled corner $c_1$ incident to it in such a way that the increasing order of the labelings coincide with the positive sence of the orientation. Let $F_{2}$ be a face incident to $c_1$, but also incident to another corner, say $c_2 \in \partial{F_1}$. Since each vertex has to have a unique label assigned to it we must to ensure that the label assigned to the corner $c_2\in \Gamma$ when we label the vertex adjacent to $F_{2}$, as especified above, is equal to the one assigned to it from the label of it as a vertex incident to $F_{1}$. We shall see that this is the case, but for the sake of readability, we will leave the proof of that to the end, and then continuing the argumentation assuming it. 

That procedure stops at some point since we have a finite number of faces, each one with only $m$ vertex adjacent to it. In that way we have constructed a surjective map 
$N:V_{\Gamma}\rightarrow \{1, 2, \cdots, m\}$. And at each blue face the indices $1, 2, \cdots, m$ appears at this order but in reverse sense of the edges orientation (recall that the edges are oriented kepping pink faces on its left side). 

But can occur that one index $k\in\{1, 2, \cdots, m\}$, or actually more than only one, do not be attained by a corner through the map $N$, i. e., so that $N^{-1}(k)$ concists only by $2$ valent vertices of $\Gamma$.

If that was not the case, then $N$ defines an admissible vertex labeling to $\Gamma$ since by construction a label $j$ is assined to only one vertex of each pink face and we have $d$ faces, furthemore, if $e$ is the valence of a vertex with label $j$ there are exacle $\frac{e}{2}$ pink face incident to it.

On the other hand, let 
$M\subset \{1,  2, \cdots, m\}$ with $|M|=p<m$ be the subset of the labelings $k\in \{1,  2, \cdots, m\}$ such that $N^{-1}(k)$ is made up only by $2$-valent vertices.
Then we can erase from the enriched graph all the vertices with label in $M$ and in the sequel to repeat the procedure of the construction of $N$ presented above wth the label set $\{1, 2, \cdots, m-p\}$.
Thus we will get a vertex labeling that tags more than one corner of the graph with the same label, for at least one label into $\{1, 2, \cdots, m-p\}$. For the same reason given above, that labeling is admissible. 

Now, let's prove the part left about the (global) consistency of the procedure presented above to construct a vertex labeling. 

Let $\{e_{j}^{1}\}_{j=1}^{x}\subset E(\Gamma^{\ast})$ and $\{e_{j}^{2}\}_{j=1}^{y}\subset E(\Gamma^{\ast})$ be the sets of edges of the bipartite dual graph $\Gamma^{\ast}$ of $\Gamma$ made up by the edges duals to the saddle-connections adjacents to $F_1$ and $F_{2}$, respectively, that form the positive path into $\Gamma$ connecting $c_1$ to $c_2$.

Thus, we consider the subgraph $G^{\ast}\subset\Gamma^{\ast}$ formed by the collection of paths into $\Gamma^{\ast}$ that possesses the inital edge in $\{e_{j}^{1}\}_{j=1}^{x}\subset E({\Gamma^{\ast}})$ and terminal edge in $\{e_{j}^{2}\}_{j=1}^{y}\subset E({\Gamma^{\ast}})$. $G^{\ast}$ have two sets of distinguished vertices, one is the singleton $I:=\{F_{1}^{\ast}\}$ and the another one is the subset 
$O\subset V(\Gamma^{\ast})$ of the vertices of $\Gamma^{\ast}$ 
duals to those blue faces that are incident to the positive path adjacent to $F_2$ joining $c_1$ and $c$. 

Note that if the cycle $\gamma:=\prod_{j=1}^{x}e_{j}^{1}\cdot{}\prod_{j=1}^{y} e_{j}^{2}$ is a separating curve of the underline surface $X$ such that the component $\Omega\subset X$ that not contains the face $F_1$ is a disk, then the defining condition of $G^{\ast}\subset\Gamma^{\ast}$ is the same that define $G^{\ast}\subset\Gamma^{\ast}$ as the subgraph of $\Gamma^{\ast}$ consisting of its part inside $\Omega$ together the edges dual to the saddle-connections into $\gamma$.

To each edge $e\in E({G^{\ast}})$ of $G^{\ast}\subset\Gamma^{\ast}$ we assign the positive integer $n-1$ where $n$ is the number of vertices over its dual saddle-connection $e^{\ast}\in E({\Gamma})$. Therefore, for each vertex of $G^{\ast}$ not being in $I\subset V(G)$ or $O\subset V(G)$ 
the sum of the numbers attached to the edges incident to it equals $m$. That is $G$, endowed with the above decribed structure, is a multi-extremal weighted graph with charge $m$. 

Let $\epsilon_{j}^{1}\geq 1$ be the number assigned to the edge $e_{j}^{1}$ for $j\in\{1,2,\cdots, x\}$ and $\epsilon_{j}^{2}\geq 1$ be the number assigned to the edge $e_{k}^{2}$ for $k\in\{1,2,\cdots, y\}$.

If $\sum_{j=1}^{x}\epsilon_{j}^{1}=\sum_{j=1}^{y}\epsilon_{j}^{2}$, the positive paths $\prod_{j=1}^{x}(e_{j}^{1})^{\ast}$ and $\prod_{j=1}^{y} (e_{j}^{2})^{\ast}$ from $c_1$ to $c$ have the same number of vertex on it, therefor the labeling atributed to $c$ by the labeling of the vertices adjacents to $F_{2}$, as described previously, will agree with the one assigned by the labeling of the vertices that are incident to $F_1.$

But, 
{Proposition $\ref{vertex-capacity-fn}$} assure the expected equality between the numbers $\sum_{j=1}^{x}\epsilon_{j}^{1}$ and $\sum_{j=1}^{y}\epsilon_{j}^{2}$, sice they are the input and output values of the multi-extremal weighted graph with constant capacity $m$, $G$.


\begin{figure}[H]
 \begin{center}
 \subfloat[]
{\includegraphics[width=4cm,height=4.2cm]{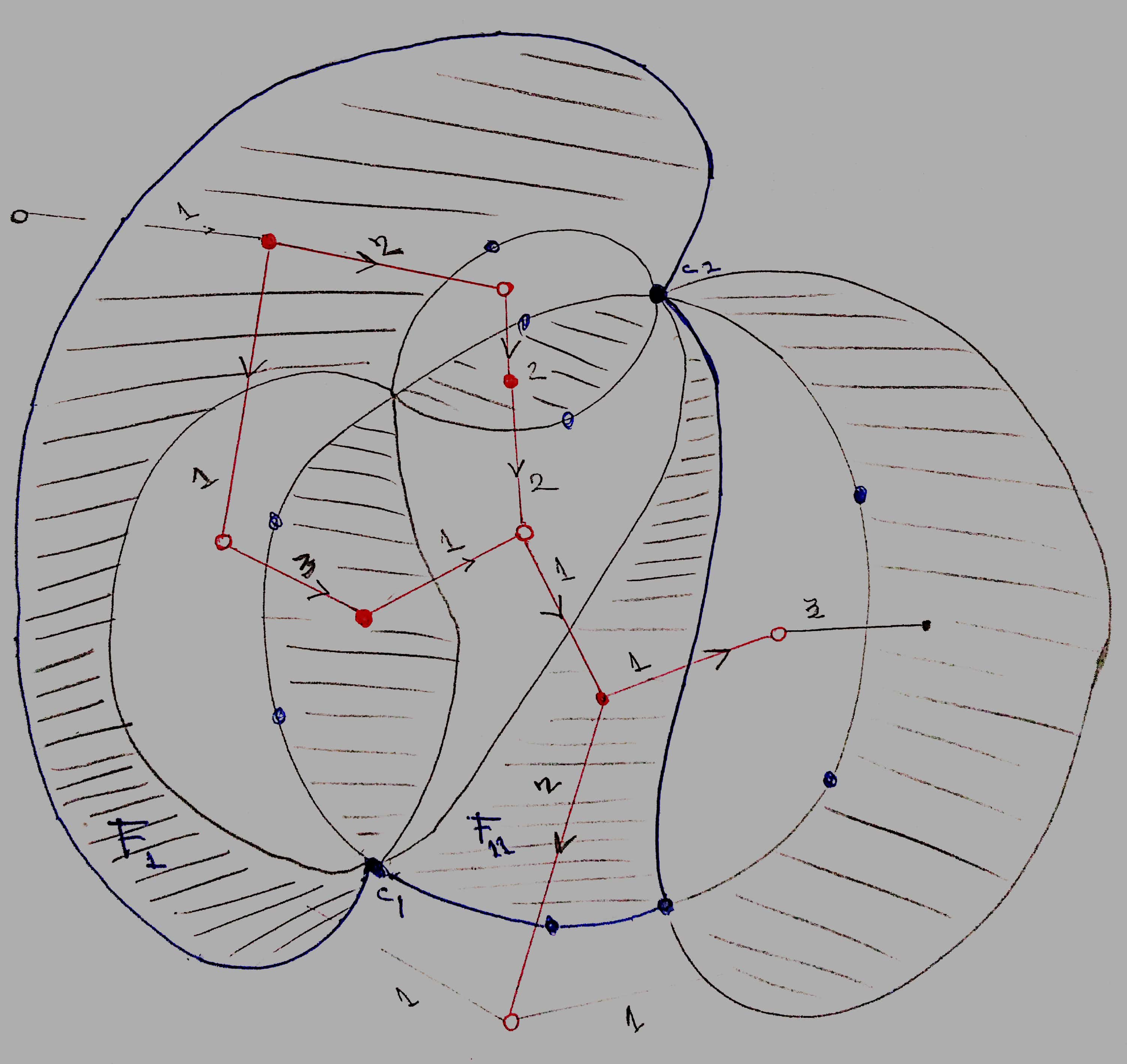}}
\quad
\subfloat[]
{\includegraphics[width=3.5cm]{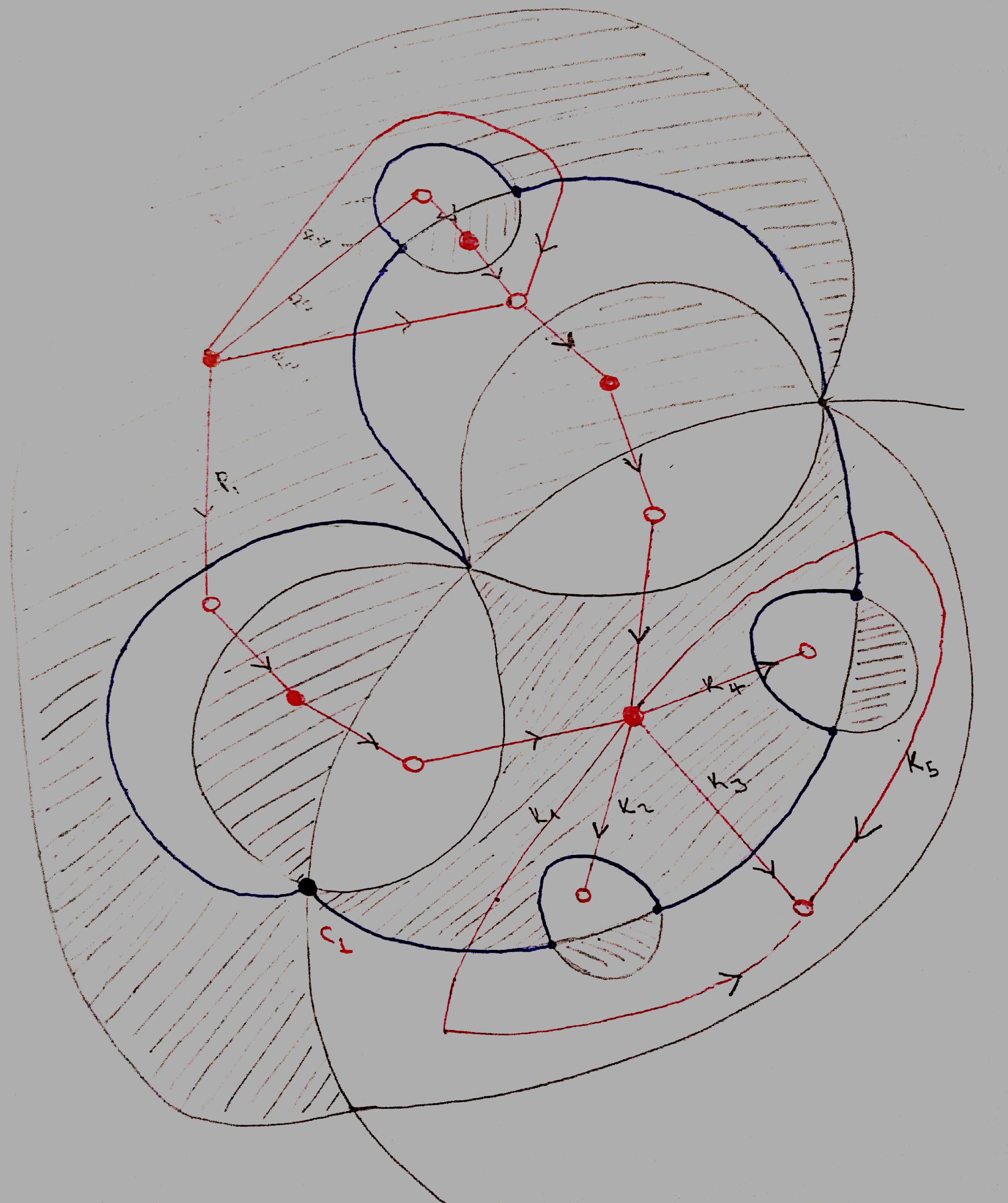}}
\caption{}
\end{center}
\end{figure}

\end{proof}

Then we are done.
\end{proof}

For reference, we will state the theorem below, which is a key feature of the second half of the previous proof.

\begin{thm}\label{gbg-adm}
Every balanced graph is admissible.
\end{thm}

Now, notice that the admissible vertex labeling depends on the matching realized to enrich the balanced graph. So a balanced graph (i.e., without $2$-valent vertices) can be the pullback graph of more than one branched cover, but all being of the same degree. See the example given below:

\begin{ex}
Distincts matchings on the same balanced graph:

\vspace{-.5cm}
 \begin{figure}[H]
 \begin{center}
       \subfloat[enriching a balanced graph from a perfect matching]
    {{\includegraphics[width=4cm]{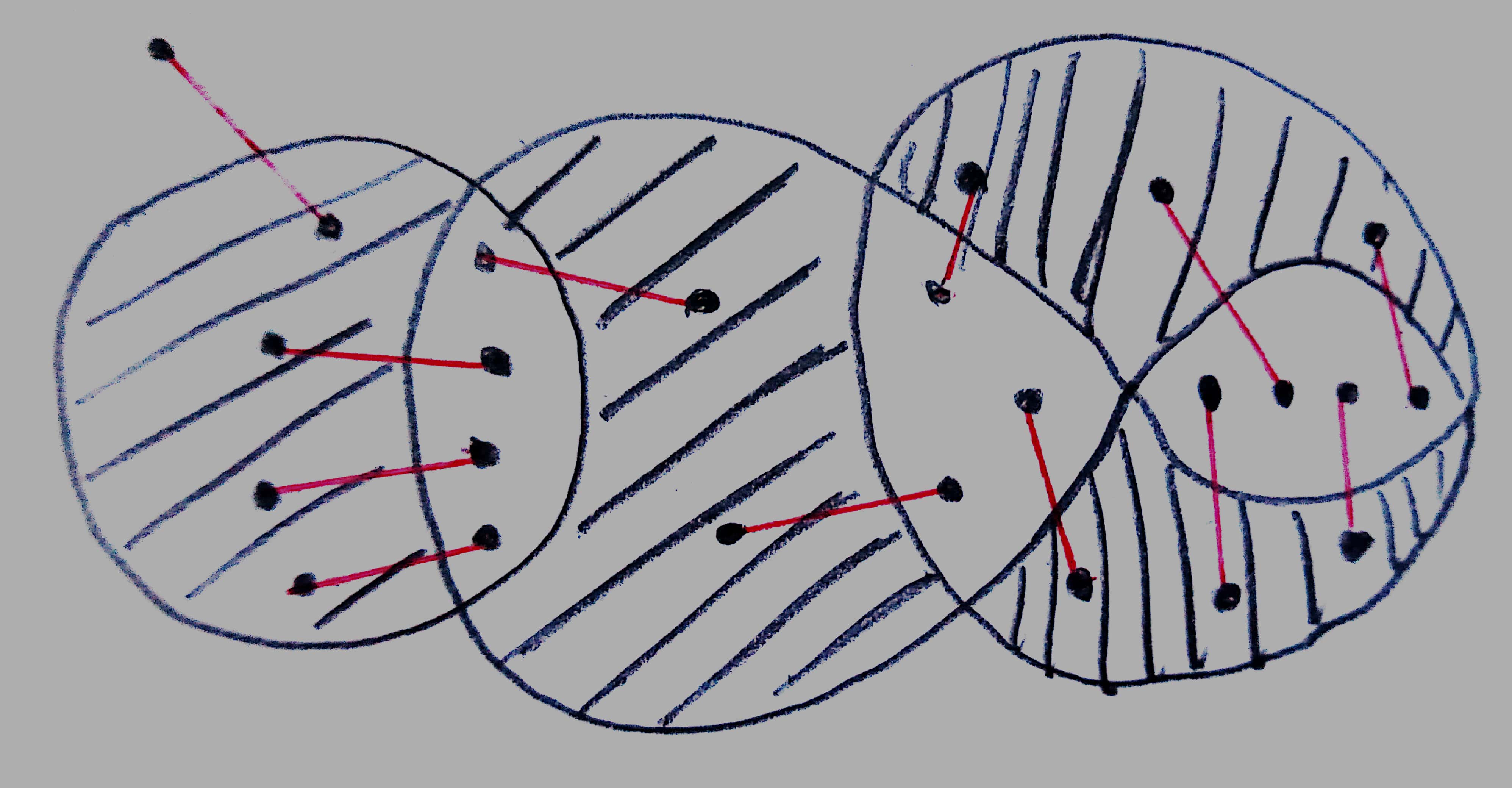}}}\quad
    \subfloat[admissible graph from (a)]  
    {{\includegraphics[width=4.25cm]{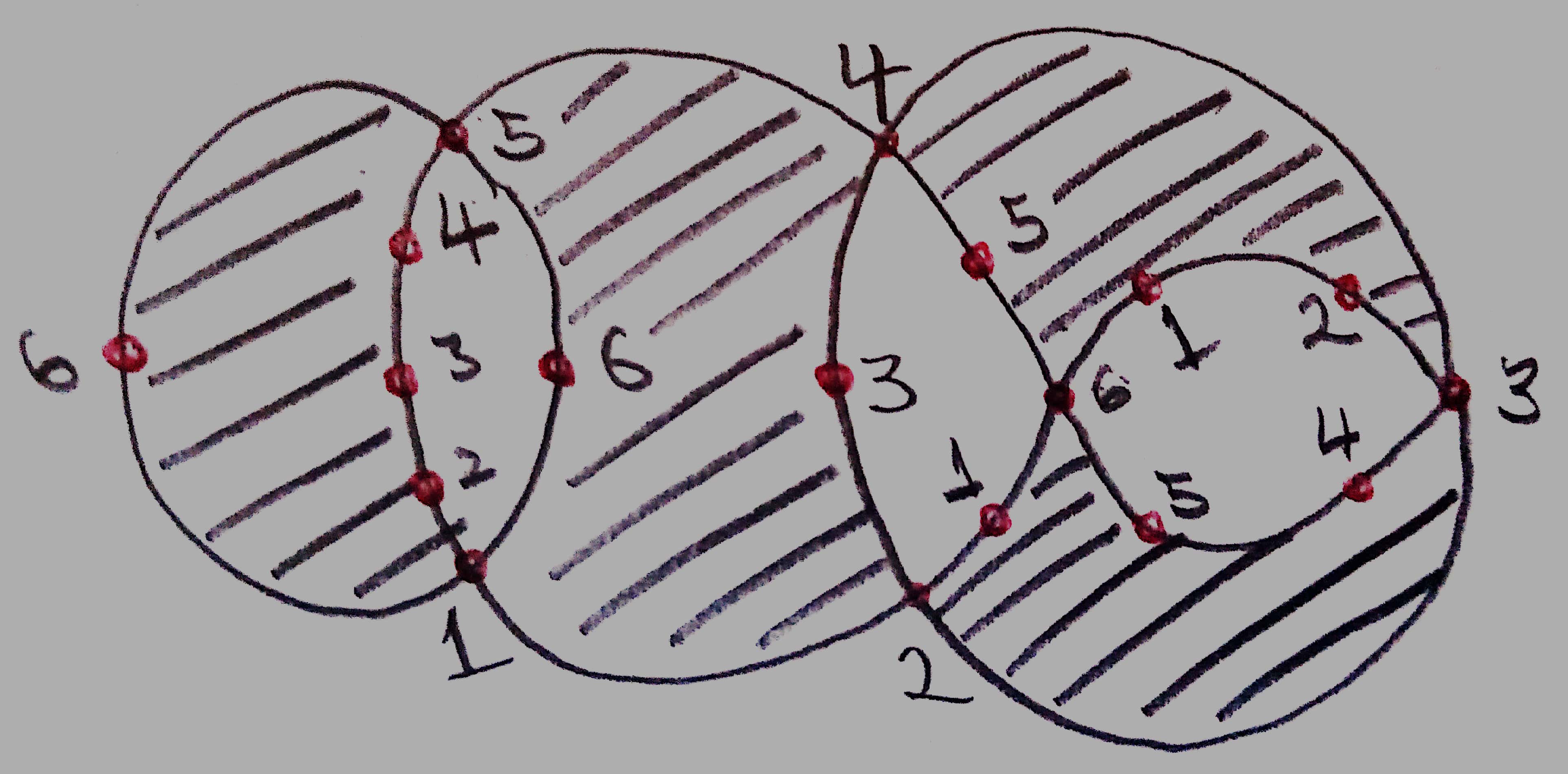} }}\quad
 
 \vspace{-.3cm}
       \subfloat[enriching a balanced graph from a perfect matching]
    {{\includegraphics[width=3.45cm,height=2.5cm]{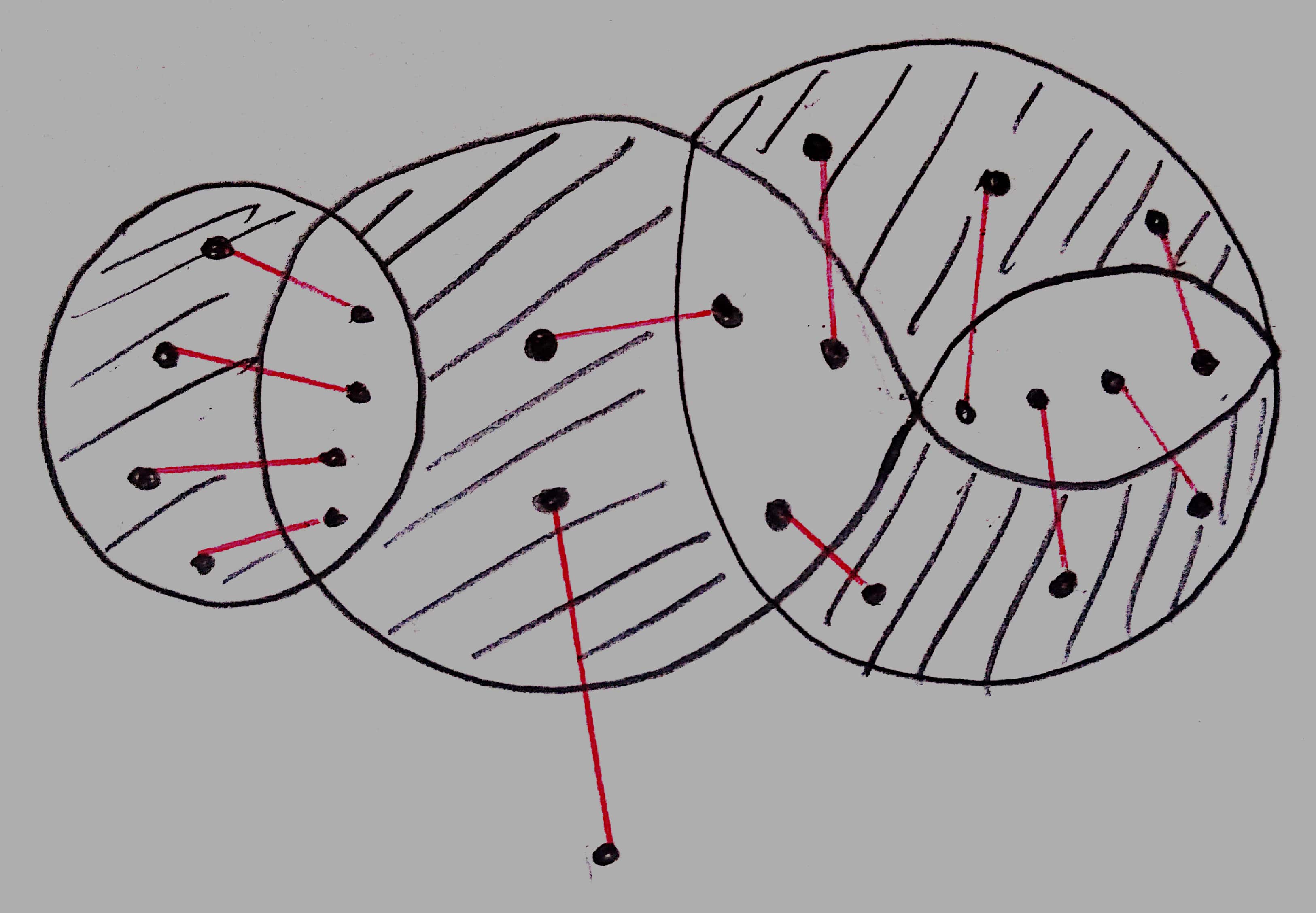} }}    
    \quad
    \subfloat[``superfluous'' admissible labelling]
    {{\includegraphics[width=3.9cm,height=2.5cm]{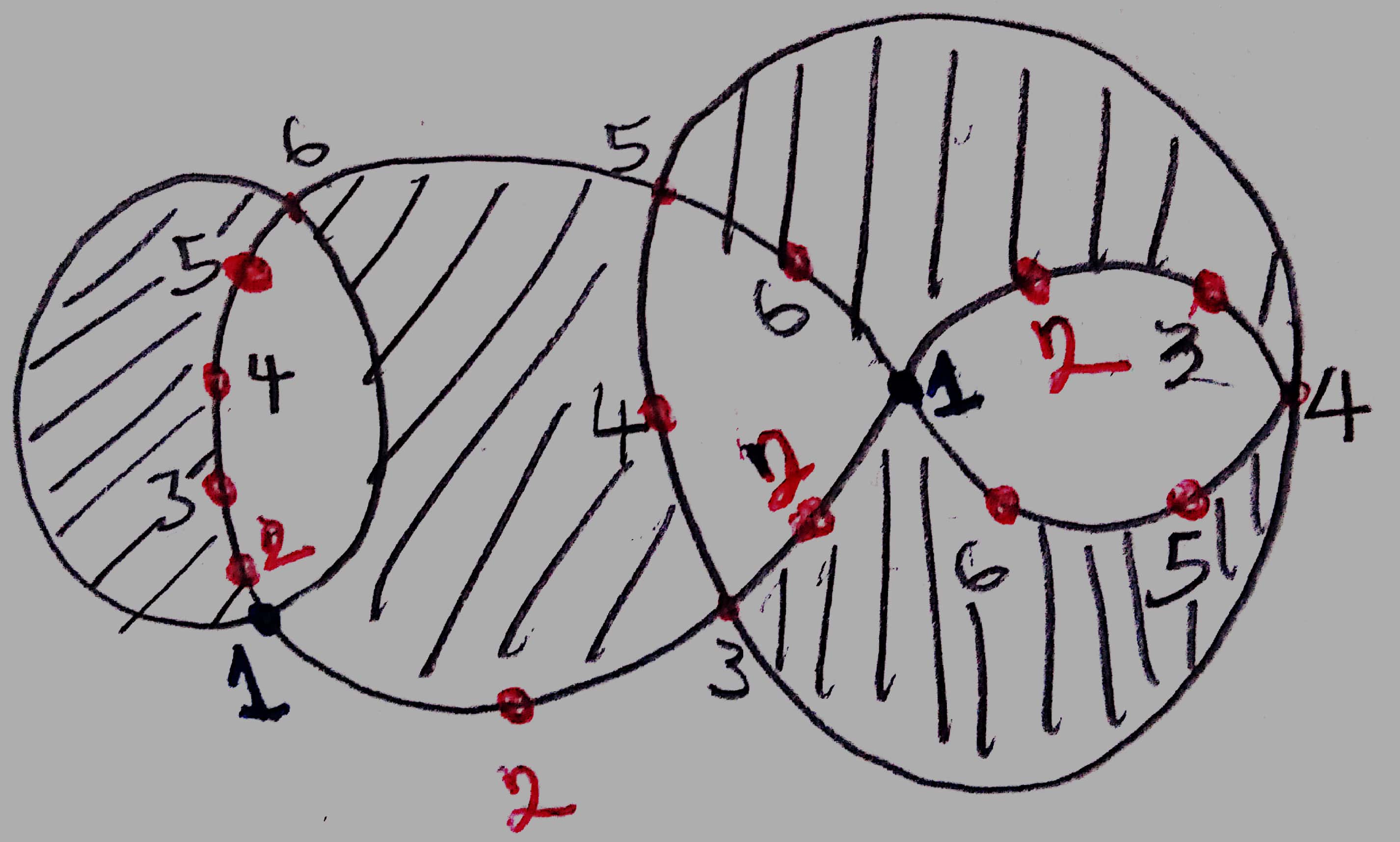} }}\quad
    \subfloat[admissible graph from (b)]
    {{\includegraphics[width=4cm,height=2.5cm]{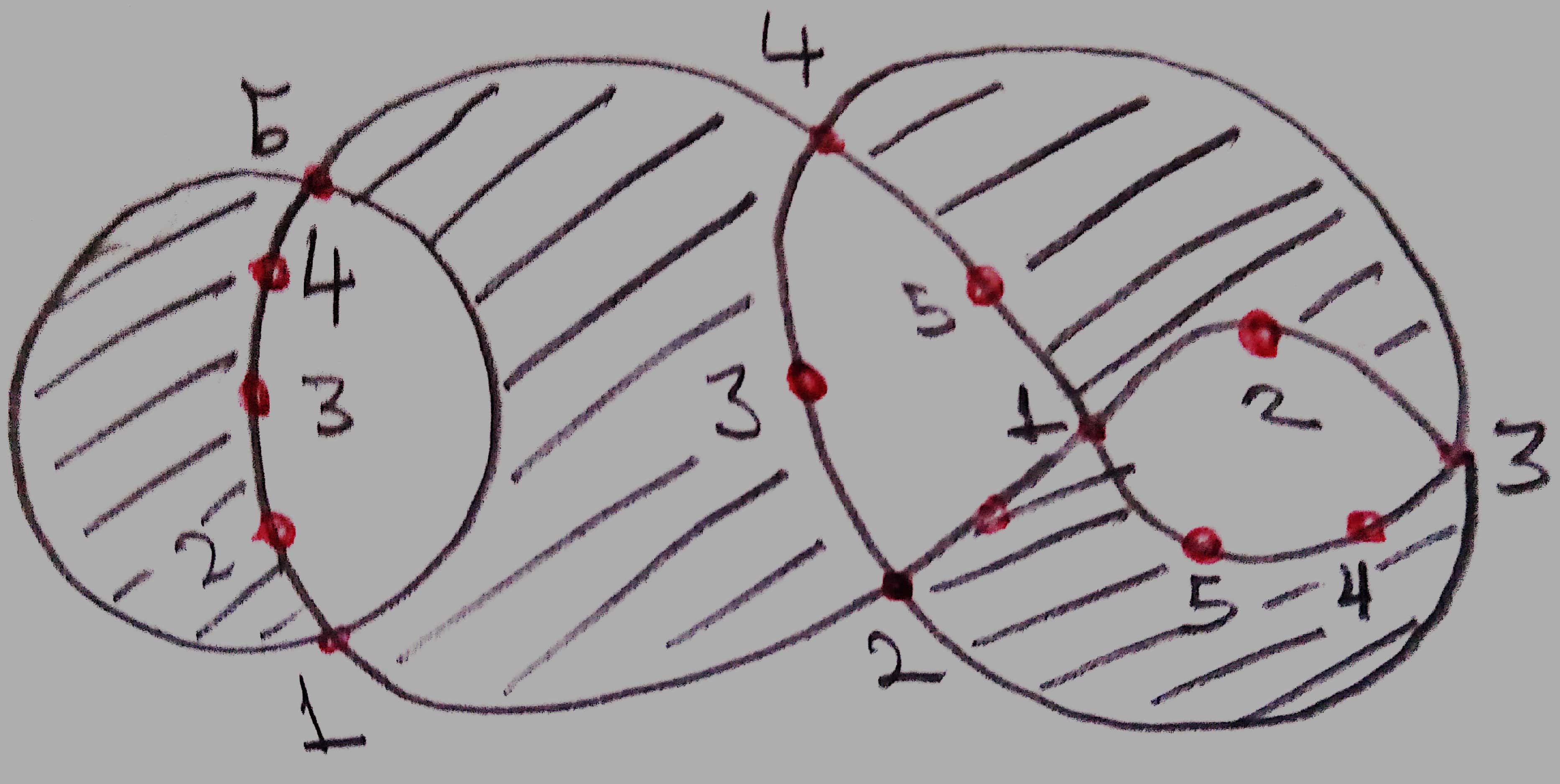} }}
    
    \vspace{-.5cm}
      \caption[matching on a balanced graph]{}
    \end{center}
\end{figure}   
\end{ex}

\vspace{-.5cm}
\begin{ex}[another example]\hspace{7cm}

\vspace{-.5cm}
\end{ex}
 \begin{figure}[H]
 \begin{center}
       \subfloat[enriching a balanced graph from a perfect matching]
    {{\includegraphics[width=4cm,height=3cm]{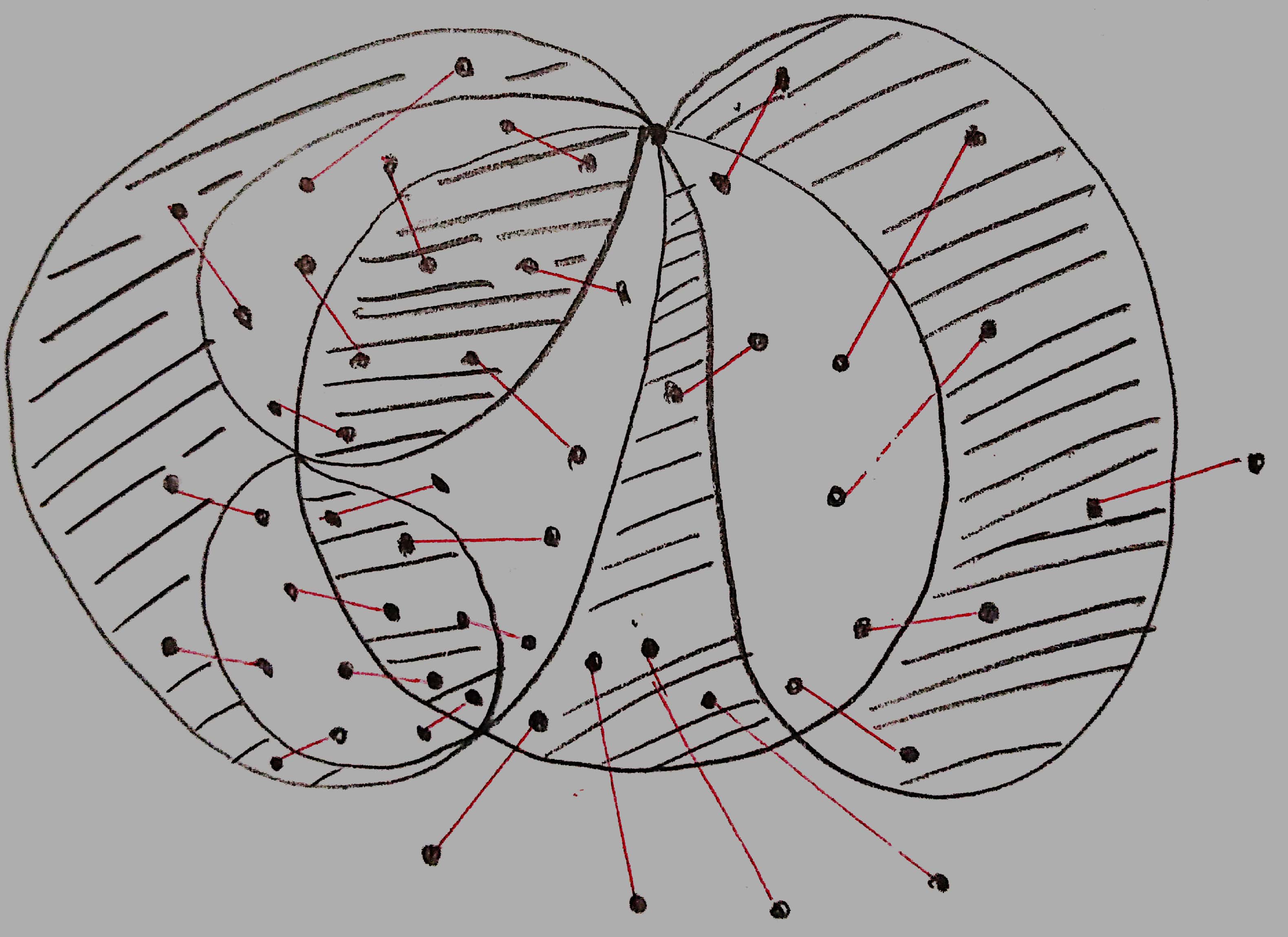} }}    
    \qquad
    \subfloat[``superfluous'' admissible labelling]
    {{\includegraphics[width=4cm,height=3cm]{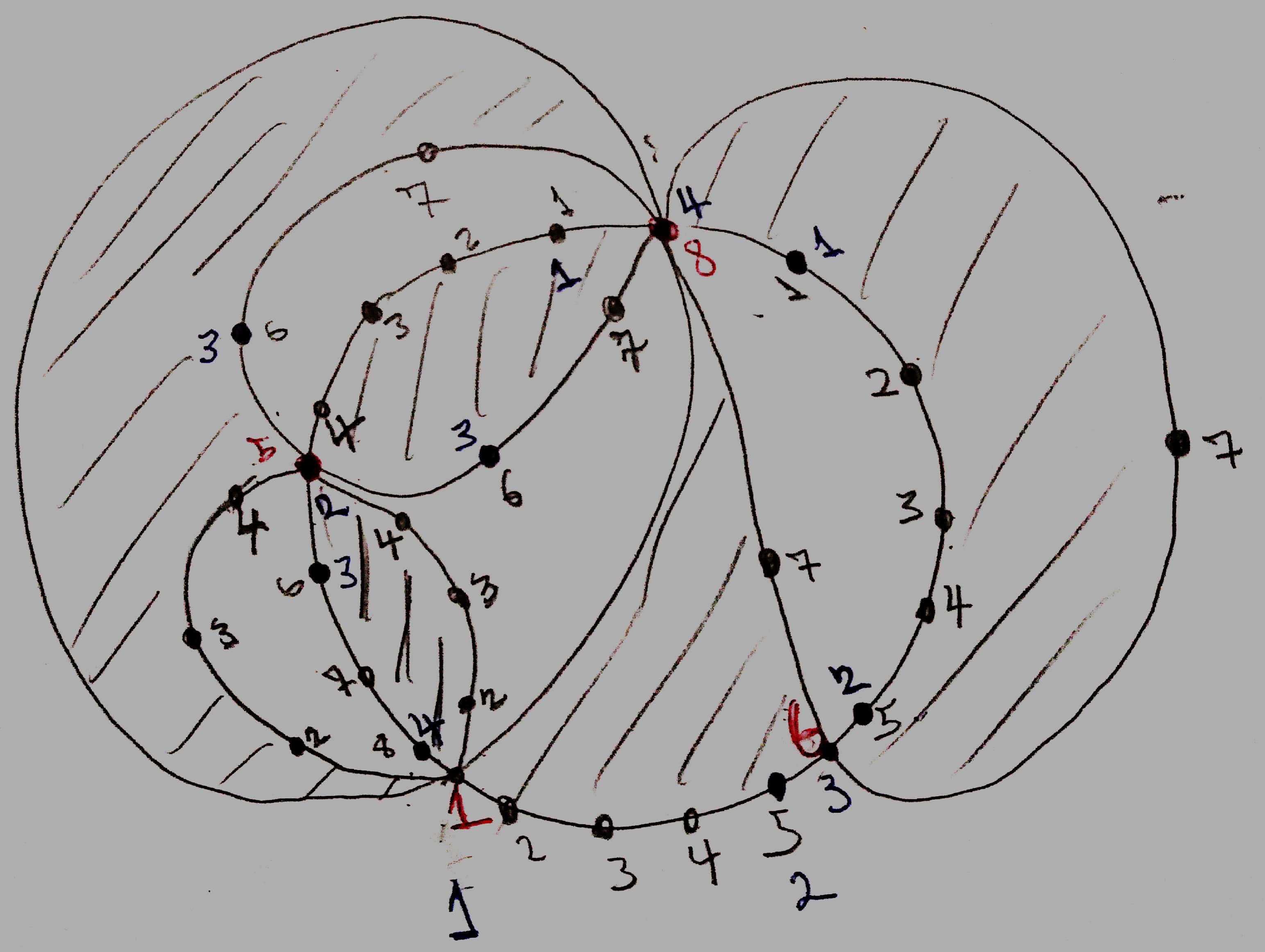} }}\quad
    \subfloat[admissible graph from (b)]
    {{\includegraphics[width=4cm,height=3cm]{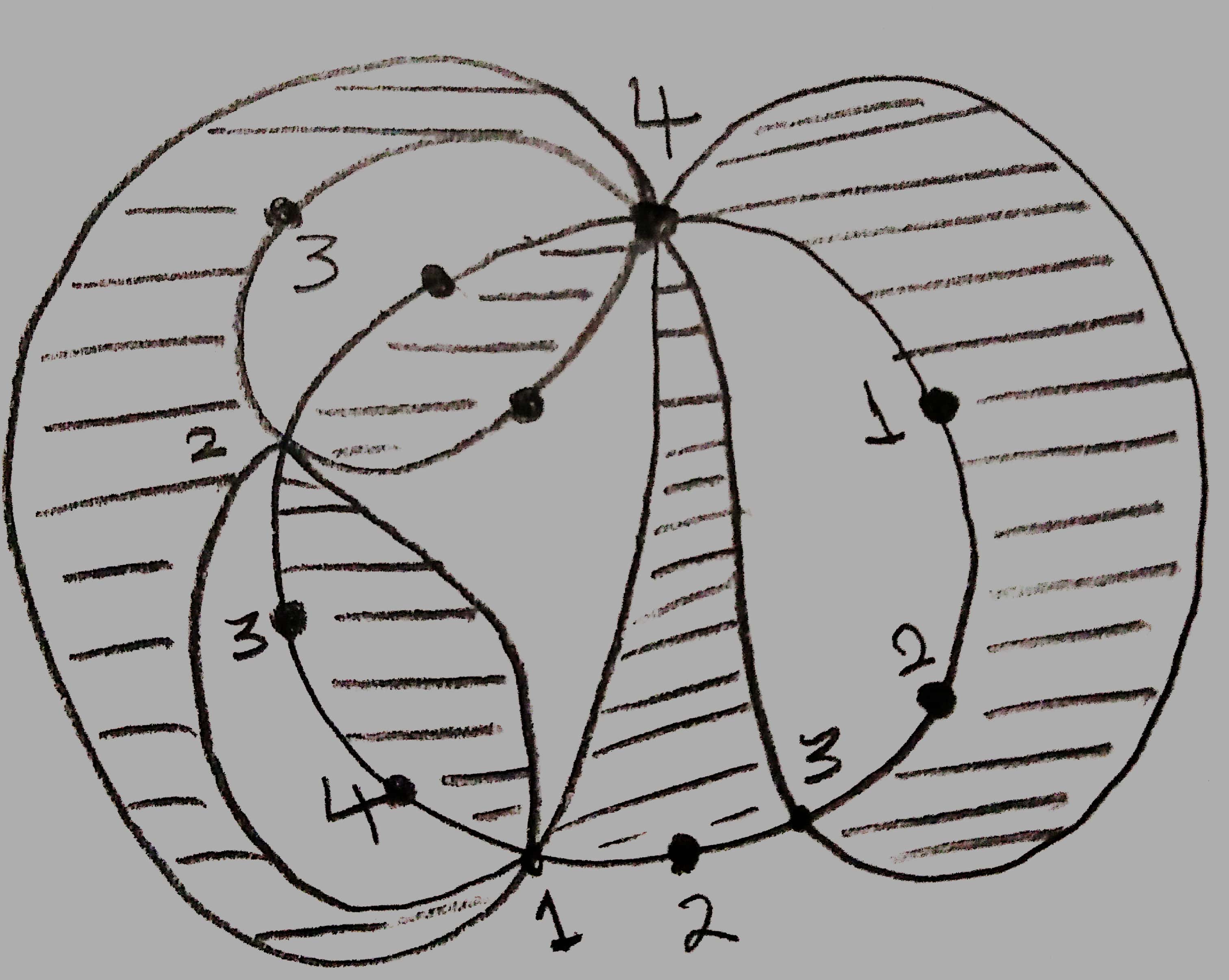} }}
    
    \vspace{-.25cm}
    \caption[matching on a balanced graph]{}
    \end{center}
\end{figure}

\section{Rational functions with real critical points}\label{2.5-sect}

A \emph{rational function} of the Riemann sphere is called to be \emph{real} it is a quotient of polynomials with real coefficients.


That class of real functions with real critical points has a \emph{canonical post-critical curve}, namely the real  line $\overline{\R}\subset\overline{\C}$, since $f(\overline{\R})\subset\overline{\R}$ for every real rational function. 

Each rational function $f \in \R(z)$ with $C(f) \subset \R$ satisfies $f(\overline{z}) = f(z)$ for all $z\in f^{-1}(\R)$. Therefore, for that type of functions, their pullback graph $\Gamma = f^{-1}(\R)$ are symmetric 
relative to $\R$. 

Thus, by the symmetry, each pullback graph $\Gamma=f^{-1}(\R)$ is uniquely determined by its non-real edges into the upper half-plane $\mathbb{H}^{u}:=\{ x+iy \in \C; y>0\}$. Any two edges of $\Gamma$ do not intersect unless at their terminal points in $\rr$. 
Our theorem $\textrm{\ref{teo-a}}$ implies that such graphs are balanced.

\begin{ex}
Real pullback graphs of some degree $3$ rational functions: \\
$f_1 (z)=\frac{z^2 \left(-\left(\sqrt{7}+2\right) z+2 \sqrt{7}+1\right)}{\left(\sqrt{7}-4\right) z+3}$, 
$f_2 (z)=\frac{z^2 \left(\left(\sqrt{7}-2\right) (-z)+2 \sqrt{7}-1\right)}{\left(\sqrt{7}+4\right) z-3}$ and $f_3 (z)=\frac{z^3}{3 z-2}$, respectively.
\begin{figure}[H]
 \begin{center}
\includegraphics[width=11cm]{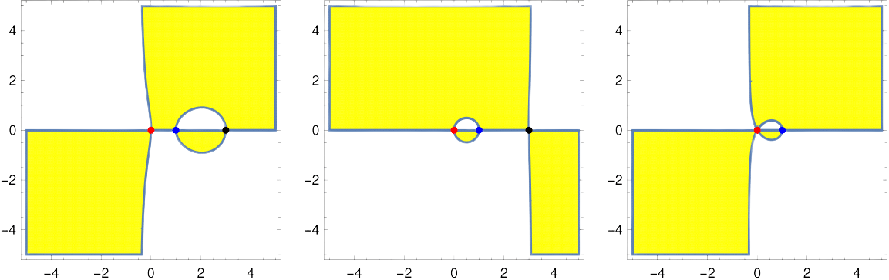}
\caption{Real Pullback Graphs}\label{rp02}
\end{center}
\end{figure}
\end{ex}

The first and the second pullback graphs in the Figure $\ref{rp02}$ corresponds to the unique two non-equivalent cubic generic real rational functions that maintains fixed the points $0,1$ and $\infty$ and it has critical points at $0,1,3, \infty$.

Such analysis of the pullback graphs’ structure provides guidance on how to construct a genuine degree $d$ real rational function with a prescribed ramification profile (see ).

Firstly, let us introduce some basic notions and results we will need for.


\begin{defn}[real globally balanced graph]\label{rgb-g}
A \emph{standard real balanced graph} is a planar globally balanced graph $\Gamma\subset\s^2$ that satisfies:
\begin{itemize}
\item[(i)]{$V(\Gamma) \subset \RPu$;}
\item[(ii)]{the $1$-skeleton of $\Gamma$ contains $\RPu$ and it is invariant under complex conjugation, $\overline{{\hspace{.1em}}^{\hspace{.35em}}}:z=x+iy\mapsto\overline{z}:=x-iy$. .}
\end{itemize}

A degree $d$ globally balanced graph is said to be \emph{real} if it is isotopic to a standard real globally balanced graph. It is called \emph{simple} 
 if it has $2d - 2$ vertices, each one of degree (valence) equals $4$. 

In a real balanced graph, the \emph{real cycle} is the one which is deformed to $\RPu$ by every isotopy attesting to its reality.
\end{defn}




\vspace{-.5cm}
\begin{figure}[H]
    \begin{center}
       \subfloat[non-simple of degree $3$]
    {{\includegraphics[width=3.6cm]{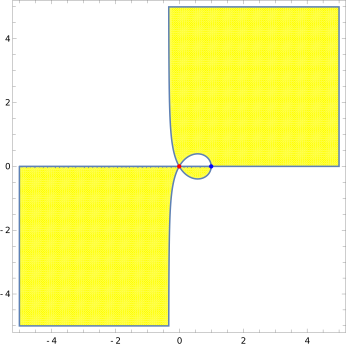} }}
        \hfill
    \subfloat[degree $3$, simple]
    {{\includegraphics[width=3.6cm]{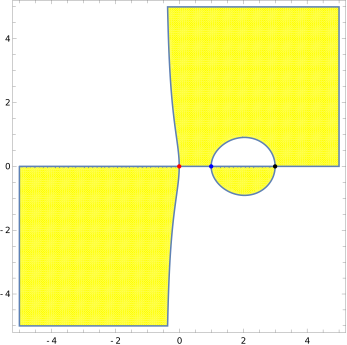} }}
\hfill
    \subfloat[non-simple of degree $7$ (polynomial)] 
    {{\includegraphics[width=3.6cm]{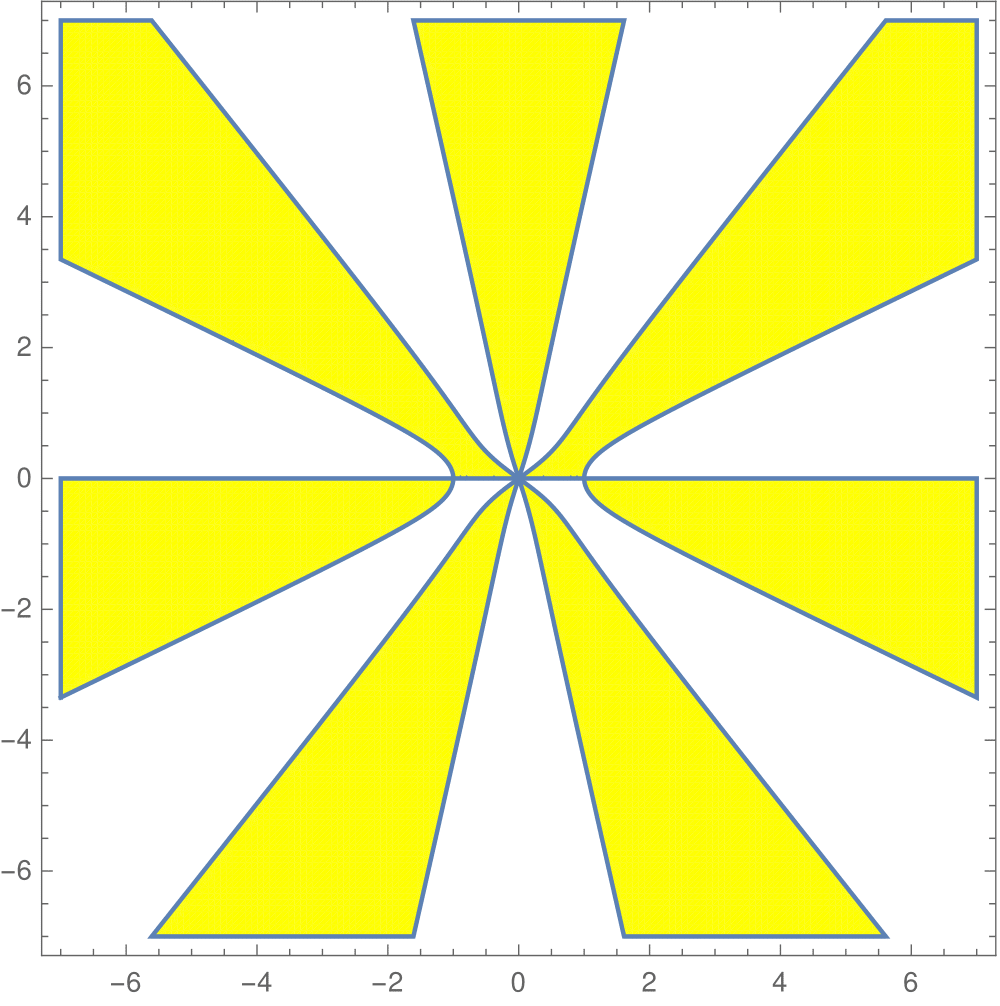} }}\hfill
    \subfloat[simple of degree $4$]
    {{\includegraphics[width=3.6cm]{deg4pbg-00.png} }}
\hfill
    \subfloat[degree $6$]
    {{\includegraphics[width=3.6cm]{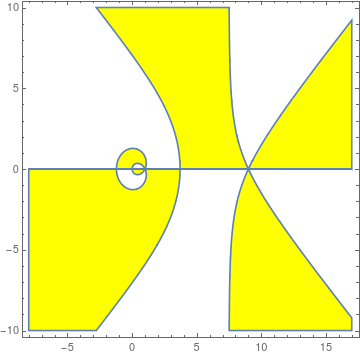} }}%
  	\end{center} 
	
	\vspace{-.5cm}
  	 \caption{real globally balanced graphs}
 \end{figure}

\end{figure}

\subsection{Constructing real rational functions}\label{crrfs}

Now, we shall se how to construct a real rational function with real critical points from a given real admissible gaph.

Let $\Gamma\subset\CPu$  be a degree $d$ real admissible graph.

\emph{Jordan-Schönflies theorem} provides a system of face embeddings, $\amalg \iota_k$, for $\Gamma\subset\CPu$ such that:
\begin{itemize}
\item[$(i)$]{ for each face $G_k \in F(\Gamma)$, for $k=1, 2, \cdots, 2d$, $\iota_k(\partial G_k)=\rr$;}
\item[$(ii)$]{$\overline{\iota_k (z)}=\iota_l (\bar{z})\quad\mbox{for all}\quad z\in |{\Gamma}|=\CPu.$}
\end{itemize}


As a result, from Corolary $\ref{bcfromg2}$, $\amalg \iota_k$ determines a rational function $F:\CPu\rightarrow \CPu$ for which holds $\overline{F(z)}=F(\overline{z})$.

Hence, we have proved that:
\begin{thm}\label{realbcfromg2}
For each real admissible graph $\Gamma$ there exist a holomorphic branched cover $\CPu\rightarrow{\CPu}$ having $\Gamma$ as a pullback graph and satisfying the identity 
\[\overline{F(z)}=F(\bar{z})\]
for all $z\in\cc$.
\end{thm}

\begin{lem}\label{lem-ratreal}
Given pollynomials $A, B, C, D\in \C[z]$ such that 
\[\dfrac{A}{B}=\dfrac{C}{D} \in \C(z)-\C\]
than there are $k\in \C-\{0\}$ such that 
\[A=k\cdot C\quad\text{and}\quad B=k\cdot{}D.\] 
\end{lem}
\begin{proof}
The zeros and poles of $f(z):=\dfrac{A}{B}$ and $g(z):=\dfrac{C}{D}$ are the same and with the same multiplicity, since they are the local degree of the two maps $f$ and $g$.  

Hence ${A}$ and ${C}$ as well as $B$ and $D$ has the same zeros with the same multiplicity, then \[A = k_1\cdot{}C\quad\text{and}\quad B = k_2\cdot D\] for some $k_1, k_2 \in\C - \{0\}$. But, $\dfrac{A}{B}=\dfrac{C}{D}$, thus $k_1 = k_2$.
\end{proof}

\begin{prop}\label{lem_relfunc}
A meromorphic function $F : {\CPu} \rightarrow {\CPu}$ satisfying for all $z \in {\CPu}$ the identity $\overline{F_{}(z)}=F(\bar{z})$ is a quotient of two polynomials with real coefficients.
\end{prop}
\begin{proof}
Let $F(z)=\dfrac{P(z)}{Q(z)}$. First, notice that for a non-constante rational fraction $f\in\C(z)$, the new one $F(z)=\overline{f(\overline{z})}$, 
is obtained by taking simply the complex conjugates of the coefficients of $f$.

Then, the relation $\overline{F(z)}=F(\bar{z})$ together the Lemma $\ref{lem-ratreal}$ implies that the coeficcients of $P(z)$ and $\overline{P(\overline{z})}$ are equals, in consequence $P$ is a polynomial with real coefficientes. We conclude same about $Q(z)$. Therefore, are real numbers all coefficients of Q.  
 

\end{proof}

The corollary below is straightforward from Theorem $\ref{realbcfromg2}$ and Proposition $\ref{lem_relfunc}$.

\begin{cor}\label{r_a_gb_r_r_f}
For each admissible real graph ${\Gamma}$ 
 there exist a \emph{real rational function} having ${\Gamma}$ as a pullback graph relative to the 
 postcritical curve $\RPu$.
\end{cor}
Now, we drive our attention to the fact that a given real rational function 
can have a non-real pullback graph. 
For a given rational function, the pullback graph depends on the isotopy type of 
the chosen post-critical curve. 
Here goes some examples:

\begin{ex}\label{isotopy-real-map}
{Some differents post-critical curves for 
\begin{eqnarray*}
f(z) = {\frac{{\frac{1}{2} \left(3-\sqrt{7}\right) z^3+\left(\sqrt{7}-2\right) z^2}}{\left(\frac{1}{2} \left(\sqrt{7}-3\right)+2\right) z-1}}
\end{eqnarray*} and its respectives pullback graphs. The critical points of $f$ are $-2,0,1$ {and} $\infty$.}

\vspace{-.5cm}
\begin{figure}[H]
    \begin{center}
        \subfloat[{pullback graph}\hspace{.5cm} $\Gamma(f,\R)$ / post-critcal curve $\R$]
    {{\includegraphics[width=3.cm]{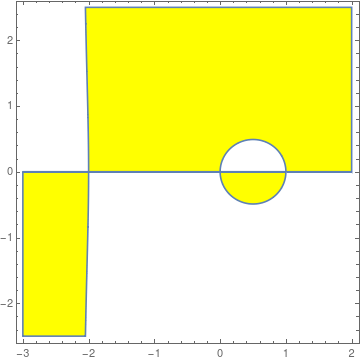} 
         \hspace{-.125cm}
    \includegraphics[width=3.cm]{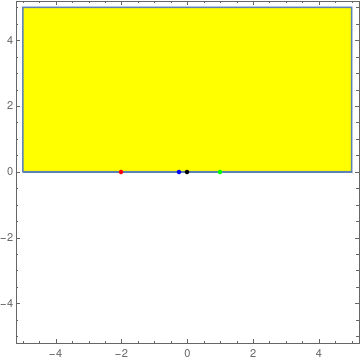} }} \hfill
     \subfloat[pullback graph $\Gamma(f,\Sigma_1)$ / post-critcal curve $\Sigma_1$] 
    {{\includegraphics[width=3.cm]{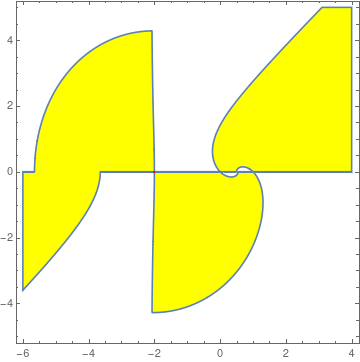}  \hspace{-.24cm} 
    \includegraphics[width=3.cm]{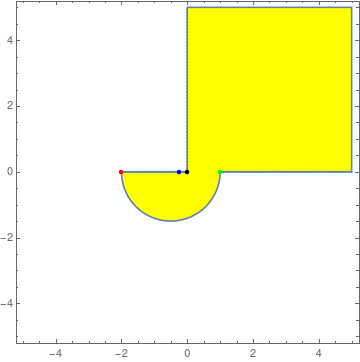} }} \hfill
     \subfloat[pullback graph $\Gamma(f,\Sigma_2)$ / post-critcal curve $\Sigma_2$]
    {{\includegraphics[width=3.cm]{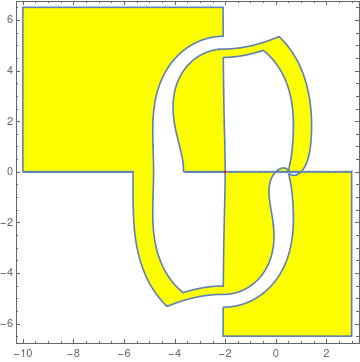}    \hspace{-.15cm}
    \includegraphics[width=3.cm]{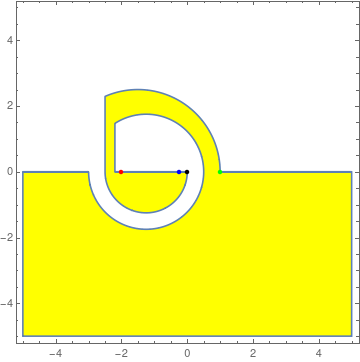} }}
    
    \vspace{-.25cm}
    \caption{real globally balanced graphs}
    \end{center}
      \label{fig:h12pullb01}	
\end{figure}

\vspace{-.5cm}
\end{ex}

\begin{lem}\label{generic-adm}
Any simple real globally balanced graph admits at least one generic admissible vertex labelling.
\end{lem}
\begin{proof}
Let $\Gamma\subset\CPu$ be simple real globally balanced graph.
Due Theorem $\textrm{\ref{teo-a}}$, $\Gamma$ admits an admissible vertex labelling over it (one for each perfect match on it as presented in $\ref{s}$, up to cyclic permutations). 
A perfect match may not produce generic labeling. In this case, we can perturb the matching to produce a new enriched graph whose labels should be pair-wise distinct (see a illustration below).
\end{proof}

\begin{lem}\label{non-istop}
For given two non-isotopic {generic real admissible graphs}, all with vertex set $\{c_1, c_2 , \cdots ,c_{m}\}\subset{\RPu}$, the corresponding real rational functions from Corollary $\ref{r_a_gb_r_r_f}$ are not equivalent. 
\end{lem}
\begin{proof}
Let $f$ and $g$ be that corresponding real rational functions. By assumption there is $\sigma\in Aut(\CPu)$ such that $g = \sigma\circ{}f$. Since ${\RPu} = f({\RPu}) = g({\RPu})$ then $\sigma({\RPu}) = {\RPu}$. Therefore, $\Gamma_{g} = g^{-1}({\RPu}) = f^{-1}(\sigma^{-1}({\RPu})) = f^{-1}({\RPu}) = \Gamma_{f}$. But this contradicts that $\Gamma_{g}$ and  $\Gamma_{f}$ are non-isotopic. 
\end{proof}

Now we shall see that every globally balanced graph is also locally balanced. 

\subsection{Local balancedness of real globally balanced graphs}\label{gbg-lbg}
Initially, we shall prove that any simple globally balanced graph satisfies the local balance condition.

\begin{mnthm}{C}[]\label{teo-c}
Every \emph{globally balanced real graphs} is \emph{locally balanced}. 
\end{mnthm} 
\begin{proof}

Let $\Gamma\subset{\CPu}$ be a simple standard globally balanced real graph with a \textcolor{deeppink}{A}-\textcolor{blue}{B} alternating face coloring and $\gamma$ a positive cycle of $\Gamma$. $\textcolor{deeppink}{A}_{\gamma}$ and $\textcolor{blue}{B}_{\gamma}$ are the numbers of \textcolor{deeppink}{A} faces and \textcolor{blue}{B} faces inside $\gamma$. 

Being $\Gamma$ a real simple globally balanced graph, each face of it have at least one of its boundary edges contained into $\overline{\R}$, we refer to such a kind of edge as real edges. By the alternating property of the face coloring each \textcolor{blue}{B} face 
possesses a companion \textcolor{deeppink}{A} face sharing the same real edges. Since $\gamma$ keeps only \textcolor{deeppink}{A} faces adjacent to its left side, for each \textcolor{blue}{B} face $F_{\textcolor{blue}{B}}$ in the interior of $\gamma$ its companion $\textcolor{deeppink}{A}$ face $F_{\textcolor{deeppink}{A}}$ is also inside $\gamma$. And, for the same reason, must there exist at least one more $\textcolor{deeppink}{A}$ face adjacent to those nonreal edges of those \textcolor{blue}{B} faces inside $\gamma$. Therefore, $\textcolor{deeppink}{A}_{\gamma}\geq \textcolor{blue}{B}_{\gamma}+1$.

We conclude that $\Gamma$ is locally balanced.

\vspace{-.25cm}
\begin{figure}[H]
 \begin{center}
       \subfloat[degree $9$]
    {{\includegraphics[width=3cm]{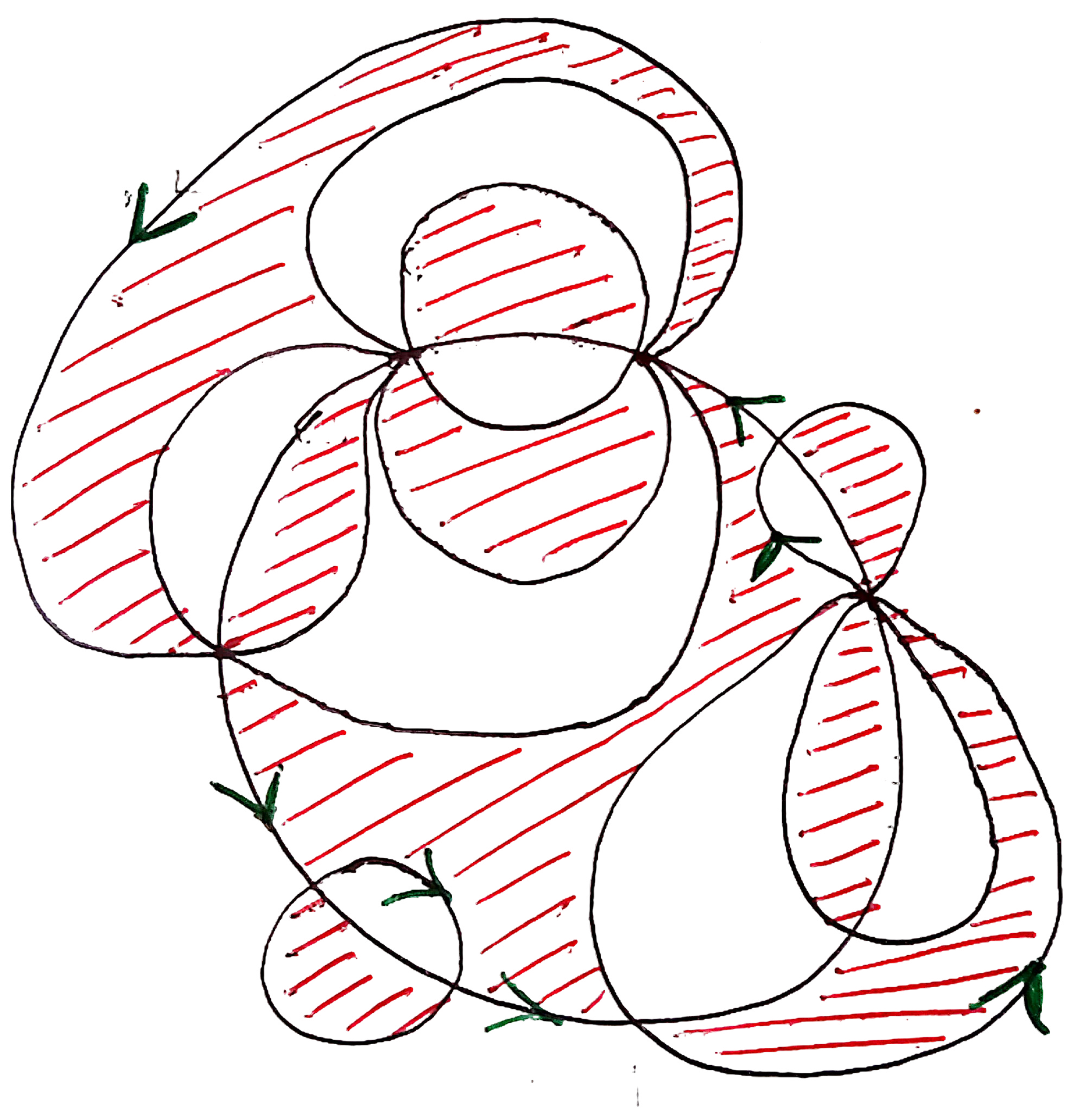}}}
        \qquad  \qquad  \qquad
    \subfloat[degree $8$ ]
    {{\includegraphics[width=3cm, height=3cm]{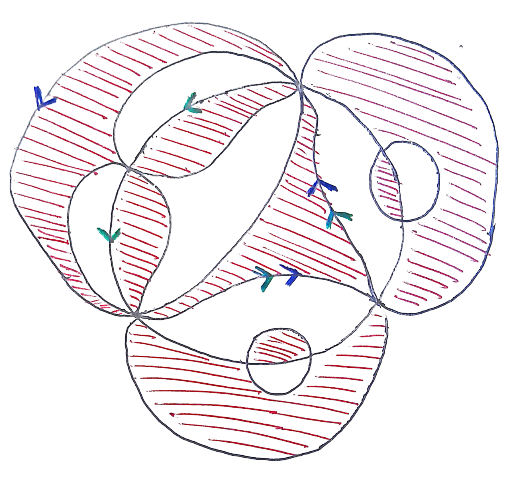} }}
      \qquad  \qquad  \qquad
    \subfloat[degree $6$]
    {{\includegraphics[width=3cm]{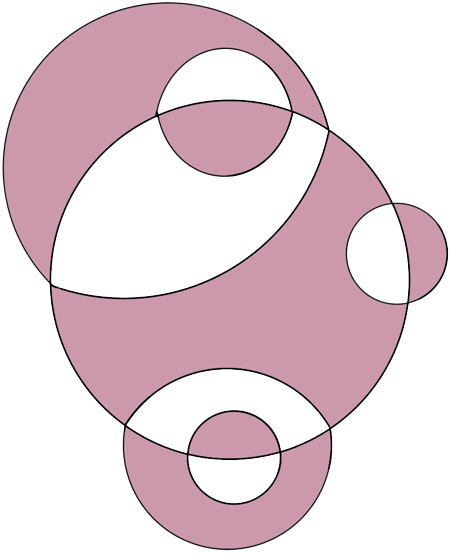} }}
  
  \vspace{-.25cm}
    \caption{real globally balanced graphs}
\end{center}\end{figure}
  
  \vspace{-.5cm}

If $\Gamma$ is standard non-simple $\gamma$, a positive cycle, may contains \textcolor{blue}{B} faces that are not incident to the real cycle (but with all their corners in the real cycle). Then, each such a \textcolor{blue}{B} face demands a $\textcolor{deeppink}{A}$ face to separate it from $\gamma$. And due to the alternation of the face coloring it must exist at leas one more $\textcolor{deeppink}{A}$ face to separate those \textcolor{blue}{B} faces inside $\gamma$, therefore $\textcolor{deeppink}{A}_{\gamma}\geq \textcolor{blue}{B}_{\gamma}+1$. 

The graph structure and face coloring of a cell graph are unaffected by an isotopic distortion, hence the results remain true for any real globally balanced graph, and we are done.
\tred{imagem}
     
\end{proof}

Now, Theorem $\textrm{\ref{teo-a}}$ give us the following corollary.

\begin{cor}\label{cor-realgb-is-lg}
For each real globally balanced graph $\Gamma$ with a admissible labelling there is a real rational function with real critical points whose pullback graph relative to $\RPu$ is $\Gamma$.
\end{cor}

Therefore, resorting to Lemma $\ref{non-istop}$ we infer 
\begin{prop}\label{prop=pre-c}
Given $C = \{c_1, c_2, \ldots, c_n\}\subset\RPu$, the number of equivalence classes of real rational function with $C\subset\RPu$ as the set of critical points is at least the number of standard real globally balanced graph.
\end{prop}

\subsection{{A combinatorial problem}}

\begin{probl}
Consider two integers $d\geq 2$ and $2\leq n\leq 2d-2$. Given $n$ points, say $\{c_1 < c_2 < \cdots < c_n\}\subset \RPu$ (with the cyclic order induced by the standard order of $\R$) in $\RPu$ and a list of integers numbers  $\textbf{a}=(a_1, a_2, \cdots ,a_n)\in\Z^n$, 
such that $1\leq a_k \leq d-1$ and $\sum_{1\leq k \leq n} a_k = 2d-2$.

Under the following constraints $\ref{ncross}$, how many ways of connecting the points $\{c_1 < c_2 < \cdots < c_n\}\subset \RPu$ by Jordan arcs are there ?

\vspace{-.5cm}
\begin{itemize}\item[]{\paragraph{Non-crossing pairing conditions: 
\vspace{-.25cm}}\begin{enumerate}\label{ncross}
\item[$(p.1)$]{the arcs are in $\mathbb{H}^{u}$ with endpoints in $\RPu$}
\item[$(p.2)$]{there is no intersection between those arcs, except possibly at their endpoints. ;}
\item[$(p.3)$]{There are $a_k$ arcs reaching $c_k$, but none of them connecting $c_k$ to itself.}\end{enumerate}}\end{itemize}

 \end{probl}


\begin{defn}[non-crossing pairing]
A \bif{non-crossing pairing of type} $(d, \textbf{a})$ is each solution for that problem, i.e., for each way of connecting the points  $\{c_1 < c_2 < \cdots < c_n\}\subset \RPu$ observing 
the constraints $\ref{ncross}$.
\end{defn}

We will relate this problem to a standard theme in enumerative combinatorics, which has a broad field of applications through mathematics such as commutative algebra, representation theory, intersection theory, degeneracies and permutations, for instance.


\subsubsection{SemiStandard Young Tableau}\label{ssyt}

Given an integer $d\geq 2$ and a $n$-upla $\textbf{a}=(a_1,a_2,\cdots , a_n)\in\Z^{n}$ of positive integers such that $a_1 + a_2 + \cdots + a_n = 2d - 2$.  A \emph{SemiStandard Young Tableau} of shape $2 \times (d - 1)$ and \emph{weight} $\textbf{a}$ is the filling of a tableaux made up by two rows of $d-1$ boxes with the integers in the set $\{1, 2, \cdots, n\}\subset\Z$ in such a way that:
\begin{itemize}
\item{the integer $i\in\{1,2, \cdots , n\}$ ocurrs $a_i$ times;}
\item{the integers increase weakly across each row from left to right and increase strictly down in each column.}
\end{itemize}
The pair $(d, \textbf{a})$ is the type of   \emph{SemiStandard Young Tableau}. We will use the abridged phrase, \emph{SSYT $(d, \textbf{a})$}, to denote a \emph{SemiStandard Young Tableau} of type $(d, \textbf{a})$.
\begin{ex}\begin{itemize}\item{\emph{SSYT} $(5; (2, 1, 1, 2, 1, 1)):$ \begin{ytableau}
1 & 1 & 4 & 4\\
2 & 3 & 5 & 6\\
\end{ytableau}\quad\begin{ytableau}
1 & 1 & 3 & 4\\
2 & 4 & 5 & 6\\
\end{ytableau}\quad\begin{ytableau}
1 & 1 & 2 & 3\\
4 & 4 & 5 & 6\\
\end{ytableau}}
\item{{ \emph{SSYT} $(8; (1, 3, 3, 5, 2)):$ \begin{ytableau}
1 & 2 & 2 & 3& 3 & 4 & 4\\
2 & 3 & 4 & 4 & 4 & 5 & 5\\
\end{ytableau} \quad  \begin{ytableau}
1 & 2 & 2& 3& 3& 3 & 4\\
2 & 4& 4& 4 & 4 & 5 & 5\\
\end{ytableau}}}
\end{itemize}
\end{ex}

As we can see above, given $(d, \textbf{a}) \in \N\times \N^n$ such that $a_1 + a_2 + \cdots + a_n = 2d - 2$  may exists more than only one SSYT of type $(d, \textbf{a})$. 

\begin{defn}
For an integer $d\geq 2$ and a $n$-upla $\textbf{a}=(a_1,a_2,\cdots , a_n)\in\Z^{n}_{\geq 1}$ of positive integers such that $a_1 + a_2 + \cdots + a_n = 2d - 2$, the \bif{Kostka Number} of shape $2 \times (d - 1)$ and type $\textbf{a}$, $K(d,{\textbf{a}})$, 
is the number of all \emph{SemiStandard Young Tableau} of shape $2 \times (d - 1)$ and type $\textbf{a}$.
\end{defn}

\begin{prop}\label{prop-nck1}
There is a $1$-to-$1$ correspondence between the set of all \emph{non-crossing pairing of type} $(d, \textbf{a})$ and the set of all \emph{SemiStandard Young Tableau} of shape $2 \times (d - 1)$ and type $\textbf{a}$.
\end{prop}
\begin{proof}
Given a \emph{non-crossing pairing of type} $(d, n, \textbf{a})$ we construct a \emph{SemiStandard Young Tableau} of shape $2 \times (d - 1)$ and type $\textbf{a}$ as it follows:
\begin{enumerate}
\item a tableaux made up by two rows of $d-1$ boxes;
\item fillout the first row from left to right with $a_1$ copies of the number $1$;
\item for $k\in\{2, 3, \cdots, n\}$ if $f_k$ is the number of arcs connecting $a_k$ to points $a_j$ with $j>k$ and $b_k$ is the number of arcs connecting $a_k$ to points $a_j$ with $j<k$;
\item From left to right, fill out the first row with $f$ copies of the number $k$;
\item From left to right, fill out the first row with $b$ copies of the number $k$;
\item Performing that Filling by choosing $k\in\{2, 3, \cdots, n\}$ in the crescent order we end up with a \emph{SemiStandard Young Tableau} of shape $2 \times (d - 1)$.
\end{enumerate}

 It is possible to reverse that construction. We then have a $1$-to-$1$ correspondence between those two sets. 


\end{proof}

\begin{defn}
Given a degree $d$ balanced graph $\Gamma$ with an enumeration $C=\{c_1, c_2, \cdots , c_n\}\subset\RPu$ of its corners, respecting the cyclic order 
in the counter-clockwise sense. The \emph{valence profile} of $\Gamma$, relative to that enumeration, is the ordered list $\textbf{v} = \textbf{v}_C=(d_1, d_2,\cdots , d_n)\in\Z^{n}_{\geq 1}$ where $d_i = {deg(c_i )}$. 
\end{defn}

Now we'll shall see how to build a degree $d$ standard real balanced graph $\Gamma$ with  \emph{valence profile} $\textbf{2a +2} = (2a_1+2, 2a_2 + 2,\cdots , 2a_n + 2)\in\Z^{n}_{\geq 1}$ out of a non-crossing pairing of type $(d, \textbf{a})$ for $\textbf{a} = \textbf{a} = (a_1, a_2,\cdots , a_n)\in\Z^{n}_{\geq 1}$.

\paragraph{Real balanced graph from non-crossing pairing }\label{rbg-nvp}

Consider a non-crossing pairing of type $(d, \textbf{a})$ of given $n\in\N$ points, $C=\{c_1, c_2, \cdots , c_n\}\subset\RPu$. 
The union of the (projective) real line, the arcs of that non-crossing pairing and its reflexion into the lower half-plane $\mathbb{H}^{d}:=\{ x+iy \in \C; y<0 \}$ with respect to ${\RPu}$ determines a connected graph, say $\Gamma$, having each point $c_i$ as a vertex with degree $2a_i + 2\in \N$. $\Gamma$ will have $2d$ faces. 

\begin{figure}[H]{
\centering
 \includegraphics[width=5cm]{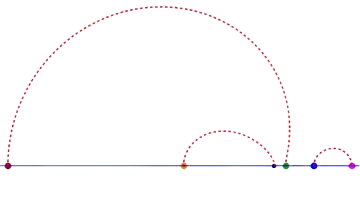}\qquad\qquad
\includegraphics[width=5cm]{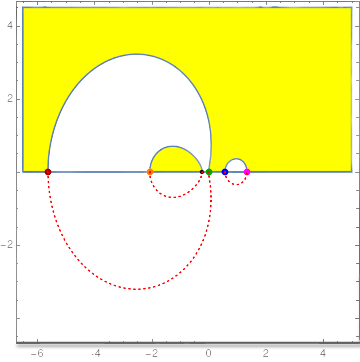}

\caption{Non-crossing pairing of type $(3, (1, 1, 1, 1, 1, 1))\in \N\times \N^{6}$ of $6$ real points}\label{match00}
}
\end{figure}


Due to the symmetry, it is then immediate that an alternating face coloring of its faces turns it into a  globally balanced graph. 

Therefore, we have just constructed a real globally balanced graph of degree $d$, say $\Gamma$, from that given non-crossing pairing of type $(d, \textbf{a})$. It's immediate that the valence profile of that real globally balanced graph is $(d, 2\textbf{a}+(2, 2, \ldots, 2))\in \N \times \N^{n} $. 


{\begin{mnthm}{B}[]\label{thm-b}
Given $d\geq 2$, for every $n\in \{2, 3, \cdots, 2d-2\}$ points in ${\RPu}$  and $\emph{\textbf{a}}= (a_1, a_2,\cdots , a_n)\in\N^{n}$, such that $\sum_{k=1}^{n} a_k = 2d-2$, there exist $K(d, \emph{\textbf{a}})$ standard real globally balanced graphs whose \emph{valence profile} is the 
integer vector $2\textbf{a} + 2$. 
\end{mnthm}}

\begin{proof}
It follows straightforwardly from Proposition $\ref{prop-nck1}$, since each \emph{non-crossing pairing} of type $K(d,{\textbf{a}})$ is associated with a unique labeled real balanced graph of degree $d$ and branch data $\textbf{a}=(a_1, a_2,\cdots , a_n)\in\Z^{n}_{\geq 1}$ (from that construction in $\ref{rbg-nvp}$).

\end{proof}

When the weight vector is $\textbf{a}= (1, 1,\cdots , 1)\in\N^{2d-2}$ the number $K(d, \emph{\textbf{a}})$ is the $d$-th \emph{Catalan Number}, $C(d) := \frac{1}{d}{2d-2\choose d-1}$. A real globally balanced graph coming through Theorem $\textrm{\ref{thm-b}}$ from a weight vector $\textbf{a}= (1, 1,\cdots , 1)\in\Z^{2d-2}_{\geq 1}$ has all its vertices with a valence $4$.

Then, as a corollary, we had achieved:

\begin{mncor}{B1}[]\label{cor-b1}
For every $2d-2$ points in ${\RPu}$ there exist $C(d) := \frac{1}{d}{2d-2\choose d-1}$ standard real globally balanced graphs with those points as vertices with valence $4$. 
\end{mncor}

Due to Proposition $\ref{prop=pre-c}$ we achieve
\begin{mnthm}{D}[]\label{teo-d}
Given a integer $d\geq 2$, for every $n\in \{2, 3, \cdots, 2d-2\}$ points in $\RPu$, $\{c_1, c_2,\cdots , c_n\}\subset\RPu$ , and vector $\emph{\textbf{a}}= (a_1, a_2,\cdots , a_n)\in\N^{n}$, such that $\sum_{k=1}^{n} a_k = 2d-2$, there exists at least $K(d, \emph{\textbf{a}})$ 
equivalence classes of real rational functions with degree $d\geq 2$ and critical points at $c_i$ with multiplicities $a_i $.
\end{mnthm}

In particular, from Collorary $\textrm{\ref{cor-b1}}$, it follows that,
\begin{mncor}{D1}[]\label{cor-d1}
For every $2d-2$ prescribed points in ${\RPu}$ there are at least $C(d) := \frac{1}{d}{2d-2\choose d-1}$ equivalence classes of generic real rational functions 
whose those chosen points are its critical points. 
\end{mncor}

\subsubsection{Proving the B. \& M. Shapiro conjecture}\label{shap-section}

In this section, we will present a new proof for the \emph{B.} \& \emph{M. Shapiro's conjecture}. 
This new proof remains at a more natural and simple level of complexity and depends much less on sophisticated non-discrete mathematical machinery than that obtained by Eremenko \& Gabrielov \cite{MR1888795}, \cite{EGSV05}, hence it is more accessible. Nevertheless, we still have to resort to Goldeberg's result \cite{Gold:91}. 


\begin{thm}[Eremenko-Gabrielov-Mukhin-Tarasov-Varchenko Theorem]\label{sha-conj}
A generic rational function $R:\overline{\C}\rightarrow\overline{\C}$ with only real critical points is equivalent to a real rational function.
\end{thm}
\begin{proof}
Given $C(R)\subset{\overline{\R}}$, from Corollary $\ref{cor-d1}$ the number of real non-equivalent classes of real rational function with critical set $C$ is at least $C_d = {\dfrac{1}{d} \binom{2d-2}{d-1}}$. But ,from Goldberg \cite{Gold:91}, the number of equivalente classes of generic rational function of prescribed critical set is at most $C_d$. Then, we are done. 
\end{proof}

\paragraph*{\begin{center}Acklowledgement\end{center}}

\vspace{-.5cm}
The author would like to express his gratitude to Sylvain Bonnot, his PhD advisor, for their thought-provoking discussions and for bringing his attention to the paper \cite{STL:15}. 
The author also thanks to the Institute for Pure and Applied Mathematics (IMPA - Rio de Janeiro) for its support and hospitality during the 2023 Post-Doctoral Summer Program. Some preliminary results of this paper, was already described in the author’s PhD thesis \cite{LA_2021} that was partially supported by CNPq, National Council for Scientific and Technological Development - Brazil.

\bibliographystyle{halpha}
\bibliography{bibliografia}  

\end{document}